\def\PaperType{arxivPaper}% arxiv version of journal or conf paper
\def\PaperVersion{final}
\def\mytitle{Over-approximating reachable tubes of linear time-varying systems}
\def\myshorttitle{Computing reachable tubes}

\def\myname{%
\ifx\shortnames\undefined Mohamed \fi
Serry%
\ifx\shortnames\undefined{ and }\else{, }\fi
\ifx\shortnames\undefined Gunther \fi
Reissig
}
\def\FirstPageFootnotes{%
\thanks{This work has been supported by the German Research Foundation
(DFG) under grant no. RE 1249/4-1.%
\ifarxivPaper \submissionNote\fi%
}
\thanks{M.~Serry is with the University of Waterloo, Dept.~of
  Mechanical and Mechatronics Eng., Waterloo, Ontario,
  Canada, mserry@uwaterloo.ca.}
\thanks{G.~Reissig is with the Bundeswehr University Munich,
  Dept.~Aerospace Eng., Institute of Control Eng., D-85577
  Neubiberg (Munich), Germany, \url{http://www.reiszig.de/gunther/}.}%
}
\def\submissionNote{This work has been accepted for publication in the
\emph{IEEE Trans. Automatic Control}. Please refer to
\url{http://dx.doi.org/10.1109/TAC.2021.3057504}
for the definite publication.
}
\def\mykeywords{Reachability, linear time-varying systems%
\ifarxivPaper%
, MSC: Primary, 93B03; Secondary, 34A60
\fi%
}

\input{GReissigIEEEclassInit}

\ifJournalPaper
\RequirePackage{generic}% information about specific journal
\fi

\ifarxivPaper\ifJournalPaper% the following line is never executed:
\documentclass{thislineisneverexecutedbutneededforarxivslatextogothrough}%
\fi\fi

\RequirePackage{GReissigMathOperators}
\RequirePackage{GReissigTextOperators}
\RequirePackage{GReissigRefstepcounter}
\RequirePackage{GReissigLanguage}
\RequirePackage{GReissigListEnvironments}
\RequirePackage{GReissigTheorems}
\RequirePackage{GReissigBibtex}
\RequirePackage{GReissigHyperref}
\RequirePackage{GReissigGeneral}\svnid{$Id: SerryReissig19a.tex 2217 2021-01-31 19:40:25Z lf1agure $}

\RequirePackage{cite}
\RequirePackage{graphicx}
\makeatletter\let\labelindent\@undefined\makeatother
\RequirePackage{enumitem}

\def\abs#1{\left|  #1 \right| }

\title{\mytitle}\author{\myname\FirstPageFootnotes}

\begin{document}
\maketitle
\begin{abstract}
We present a method to over-approximate reachable tubes over compact
time-intervals, for linear continuous-time, time-varying control systems whose initial
states and inputs are subject to compact convex uncertainty. The method uses numerical
approximations of transition matrices, is convergent of first order,
and assumes the ability to compute with compact convex sets in finite
dimension. We also present a variant that applies to the case of
zonotopic uncertainties, uses only linear algebraic operations, and
yields zonotopic over-approximations. The performance of the latter
variant is demonstrated on
an
example.
\end{abstract}


\section{Introduction}
\label{s:intro}

Reachable (or \concept{attainable}) sets and tubes are central
concepts in systems and control theory, with myriads of
applications. See,
e.g.~\cite{BasileMarro92,BlanchiniMiani08,Tabuada09,%
AlthoffBakForetsFrehseKochdumperRaySchillingSchupp19,%
KurzhanskiVaraiya14,ReissigWeberRungger17,Althoff18%
}
and the references given therein.
The efficient computation of accurate approximations of these sets is
a challenging problem whose diverse variants have been attracting
research attention for decades.
In this paper, we focus on
over-approximating reachable tubes of
linear time-varying control systems of the form
\begin{equation}
\label{eq:LTVsystem}
\dot{x}(t) = A(t) x(t) + B(t) u(t)
\end{equation}
over compact time-intervals $\intcc{t_0,t_f}$,
where $A \colon \intcc{t_0,t_f} \to \mathbb{R}^{n \times n}$ and
$B \colon \intcc{t_0,t_f} \to \mathbb{R}^{n \times m}$ are
time-varying matrices.
Both the initial state $x(t_0)$ and the input signal $u$ are
subject to
compact convex uncertainty. The problem is mathematically formalized
in Section \ref{s:ProblemStatement}.

In numerous applications it is critical to formally verify that all
solutions of the system \ref{eq:LTVsystem} always avoid certain
predefined \concept{unsafe regions}; see,
e.g.~\cite[Sect.~3]{AlthoffBakForetsFrehseKochdumperRaySchillingSchupp19}
and the references given therein.
That is, these applications require proof that the reachable
tube over the time-interval $\intcc{t_0,t_f}$ (and not only the
reachable set at some time $t \in \intcc{t_0,t_f}$) of the system
\ref{eq:LTVsystem} does not intersect any unsafe region.
As tubes cannot, in general, be determined exactly,
intersection tests need to rely on over-approximations (and not on
mere approximations) in place of the actual tubes. The
over-approximations should be as precise as possible to avoid
excessive conservatism of the verification, and need to be represented
in a form that facilitates to reliably and efficiently verify
disjointness from unsafe regions.

One of the earliest techniques of reachability analysis, the
\concept{hyperplane method}, approximates reachable sets by
intersections of supporting halfspaces and by convex hulls of the
respective support points
\cite{PecsvaradiNarendra71,BasileMarro92}. More recent
techniques rely on a variety of additional classes of sets
including, e.g.~ellipsoids, hyper-intervals, and zonotopes
\cite{KurzhanskiVaraiya14,Althoff18,%
Girard05,GirardLeGuernicMaler06,LeGuernicGirard09,%
FrehseLeGuernicDonzeCottonRayLebeltelRipadoGirardDangMaler11,%
AlthoffFrehse16,AlthoffKroghStursberg11,BotchkarevTripakis00,%
VillanuevaQuirynenDiehlChachuatHouska17,SerryReissig18aC,%
ReissigWeberRungger17,ShenScott20,ZamaniPolaMazoTabuada10,%
NedialkovJacksonCorliss99,Veliov92,%
BeynRieger07,ReissigRungger19,BakDuggirala17,Baier95%
}.
As for reachable tubes, the standard approach today is to apply the
method proposed in \cite{Girard05} or one of its extensions,
e.g.~\cite{GirardLeGuernicMaler06,LeGuernicGirard09,Althoff18,%
FrehseLeGuernicDonzeCottonRayLebeltelRipadoGirardDangMaler11,%
AlthoffFrehse16}, which compute over-approximations in the form of
finite unions of zonotopes
\cite{Girard05,GirardLeGuernicMaler06,Althoff18}
and of more general convex sets
\cite{LeGuernicGirard09,%
FrehseLeGuernicDonzeCottonRayLebeltelRipadoGirardDangMaler11,%
AlthoffFrehse16}.
As a result of such representation, disjointness from a polyhedral (or
convex) unsafe region can be verified by solving a linear (or convex)
feasibility problem. While the method is particularly efficient
and converges, i.e., it is capable of producing arbitrarily precise
over-approximations, its application is limited to the time-invariant
special case of \ref{eq:LTVsystem}. Its extension in
\cite{AlthoffKroghStursberg11} additionally allows for uncertain
coefficients $A$ and $B$ in \ref{eq:LTVsystem}, but does not converge
even if $A$ and $B$ are precisely known, in which case
\ref{eq:LTVsystem} is again required to be time-invariant.

Another prominent class of methods, \concept{ellipsoidal techniques}
\cite{KurzhanskiVaraiya14}, solve the more general problem of feedback
synthesis for linear time-varying plants with two competing
inputs. When applied to the system \ref{eq:LTVsystem}, these methods
yield a set-valued function $E$ defined on the interval
$\intcc{t_0,t_f}$ whose value at any time is a finite intersection of
ellipsoids containing the reachable set at that time as a
subset. While arbitrarily precise over-approximations are obtained
when a sufficient amount of ellipsoids is computed, the approach
suffers from two shortcomings. Firstly, the ellipsoids result from
numerically solving linear-quadratic optimal control problems derived
from \ref{eq:LTVsystem}, yet numerical errors incurred in the course
of the solution are not taken into account. Hence, mere approximations
rather than over-approximations might actually be computed.
Secondly, approximations of reachable tubes are obtained only
implicitly, as the union over $t \in \intcc{t_0,t_f}$ of $E(t)$, and
so they are disjoint from an unsafe region $R$
if and only if the graph of the set-valued map $E$ is disjoint from
the set $\intcc{t_0,t_f} \times R$. Verifying the latter condition
is a great challenge since the graph of $E$ is not, in general,
convex. The issue has so far been resolved only for the time-invariant
special case of \ref{eq:LTVsystem}; see \cite{BotchkarevTripakis00}.
Moreover, while ellipsoidal techniques have been generalized to handle
nonlinear dynamics, the extensions still suffer from both the
aforementioned shortcomings,
e.g.~\cite{VillanuevaQuirynenDiehlChachuatHouska17}.

Other approaches use differential inequalities, comparison
principles, interval arithmetic, and combinations thereof, and compute
interval over-approximations
\cite{ShenScott20,ZamaniPolaMazoTabuada10,ReissigWeberRungger17,%
SerryReissig18aC,NedialkovJacksonCorliss99}.
While these techniques may allow for uncertain coefficients $A$ and
$B$ in \ref{eq:LTVsystem} \cite{SerryReissig18aC} or even for
nonlinear dynamics
\cite{ShenScott20,ZamaniPolaMazoTabuada10,ReissigWeberRungger17,%
NedialkovJacksonCorliss99}, they are all conservative, i.e.,
arbitrarily precise over-approximations of reachable tubes cannot be
obtained, and the methods in 
\cite{ShenScott20,ZamaniPolaMazoTabuada10,ReissigWeberRungger17,%
SerryReissig18aC}
additionally suffer from both shortcomings mentioned in our discussion
of ellipsoidal techniques.
Finally, the reachable tube can also be characterized as a sublevel
set of the viscosity solution of a partial differential equation
called \concept{Hamilton-Jacobi-Bellman equation}
\cite{KurzhanskiVaraiya14}. However, solving the latter numerically
is avoided in practice as this would require discretizing the state
space and so the computational effort would scale exponentially with
the state space dimension.

To conclude, efficient methods to compute arbitrarily precise
over-approximations of reachable tubes of the system
\ref{eq:LTVsystem}, that are additionally represented in a form
suitable for formal verification purposes, are currently limited to
the time-invariant special case of \ref{eq:LTVsystem}. This is in
stark contrast to the importance of the general time-varying case
of \ref{eq:LTVsystem} in several fields of application,
e.g.~\cite{FossenNijmeijer11}.

In Section \ref{ss:MainResults:General} of this paper,
we present a method that produces over-approximations that are
convergent of first order,
does not require discretization of either the input
or the state space,
uses numerical approximations of transition
matrices rather than closed-form solutions, and assumes the ability to
compute with compact convex sets in finite dimension.
A variant that applies to the case of zonotopic uncertainties, uses
only linear algebraic operations, and yields zonotopic
over-approximations, is subsequently presented in Section
\ref{ss:MainResults:Zonotopes}.
In Section \ref{s:examples}, we demonstrate the performance of the
latter variant on an example.

\section{Preliminaries}
\label{s:prelims}

\subsection{Notation}
\label{ss:prelims:Notation}

Given two sets $A$ and $B$ and a positive integer $p$,
$B \setminus A$ and $A \times B$ denotes the relative complement of
the set $A$ in the set $B$, and the product of $A$ and $B$,
respectively, and $A^p = A \times \cdots \times A$ ($p$ factors).
$\mathbb{R}$, $\mathbb{R}_+$, $\mathbb{Z}$ and $\mathbb{Z}_{+}$ denote
the sets of real numbers, non-negative real numbers, integers and
non-negative integers, respectively, and
$\mathbb{N} = \mathbb{Z}_{+} \setminus \{ 0 \}$.
$\intcc{a,b}$, $\intoo{a,b}$,
$\intco{a,b}$, and $\intoc{a,b}$
denote closed, open and half-open, respectively,
intervals with end points $a$ and $b$,
e.g.~$\intco{0,\infty} = \mathbb{R}_{+}$.
$\intcc{a;b}$, $\intoo{a;b}$,
$\intco{a;b}$, and $\intoc{a;b}$ stand for discrete intervals,
e.g.~$\intcc{a;b} = \intcc{a,b} \cap \mathbb{Z}$,
$\intco{1;4} = \{ 1,2,3 \}$, and
$\intco{0;0} = \emptyset$.
Given any map $f \colon A \to B$, the image of a subset $C \subseteq
A$ under $f$ is denoted $f(C)$, $f(C) = \Set{ f(c) }{ c \in C }$.
We denote the identity map $X \to X \colon x \mapsto x$ by $\id$,
where the domain of definition $X$ will always be
clear form the context.

Arithmetic operations involving subsets of a linear space $X$ are
defined pointwise, e.g.
$\alpha M \defas \Set{ \alpha y }{ y \in M }$ and the Minkowski sum
$M + N \defas \Set{ y + z }{ y \in M, z \in N }$, if
$\alpha \in \mathbb{R}$ and $M, N \subseteq X$.
The convex hull of $M$ is denoted $\conv( M )$.
By $\| \cdot \|$ we denote any norm on $X$,
$\mathbb{B} \subseteq X$ is the closed unit ball w.r.t.~$\| \cdot \|$,
and the norm of a non-empty
subset $M \subseteq X$ is defined by
$\| M \| \defas \sup_{x \in M} \| x \|$.
The maximum norm
on $\mathbb{R}^n$ is denoted $\| \cdot \|_{\infty}$,
$\| x \|_{\infty} = \max \Set{ | x_i | }{ i \in \intcc{1;n} }$
for all $x \in \mathbb{R}^n$.
The Hausdorff distance $d_H$ is defined in the Appendix.

We say that a map is \concept{of class $C^k$} if it is continuous and
$k$ times continuously differentiable, $k \in \mathbb{Z}_+$.
Given a non-empty set $X \subseteq \mathbb{R}^n$ and a compact
interval $\intcc{a,b} \subseteq \mathbb{R}$, $X^{\intcc{a,b}}$ denotes
the set of all measurable maps $\intcc{a,b} \to X$. Integration is
always understood in the sense of Lebesgue.
Given norms on $\mathbb{R}^n$ and $\mathbb{R}^m$, the linear space
$\mathbb{R}^{n \times m}$ of $n \times m$ matrices is endowed with the
usual matrix norm,
$\| A \| = \sup_{\| x \| \le 1} \| A x \|$ for
$A \in \mathbb{R}^{n \times m}$.

We use the asymptotic notation $O(\cdot)$ in the usual way
\cite{deBruijn81}. In particular,
let $X \subseteq \mathbb{R}^n$,
$f \colon F \subseteq \mathbb{R} \times X \to \mathbb{R}_{+}$,
$g \colon G \subseteq \mathbb{R} \to \mathbb{R}_{+}$,
$H \colon F \times \mathbb{R}_{+} \to \mathbb{R}_{+}$ and
$a \in \mathbb{R}$ be such that
$a = \lim_{i \to \infty} s_i$ for some sequence
$(s_i,x_i)_{i \in \mathbb{N}}$ in $F$, and suppose that
$s \in G$ whenever $(s,x) \in F$.
Then $f(s,x) \le H(s,x,O(g(s)))$ as $s \to a$, uniformly in $x$, if
there exist $k \colon G \to \mathbb{R}_{+}$ and a
neighborhood $U \subseteq \mathbb{R}$ of $a$ such that
$k(s) = O(g(s))$ as $s \to a$ and
$f(s,x) \le H(s,x,k(s))$ whenever $(s,x) \in F \cap ( U \times X )$,
and similarly for $a \in \{ \infty, -\infty \}$.

\subsection{Linear Time-Varying Control Systems}
\label{ss:Systems}

\noindent
Given $u \colon \intcc{t_0,t_f} \to \mathbb{R}^m$, a map
$x \colon \intcc{t_0,t_f} \to \mathbb{R}^n$ is a \concept{solution} of
the system \ref{eq:LTVsystem} (generated by $u$) if $x$ is absolutely
continuous and
\ref{eq:LTVsystem} holds for (Lebesgue) almost every
$t \in \intcc{t_0,t_f}$.
We shall always assume that $A$ and $B$ are continuous and that $u$ is
integrable, which implies both existence and uniqueness of solutions
\cite{Lukes82}.
The \concept{general solution} of the system \ref{eq:LTVsystem} is the
map $\varphi$ defined by the requirement that for all
$p \in \mathbb{R}^n$, $s \in \intcc{t_0,t_f}$ and integrable
$u$, $\varphi(\cdot,s,p,u)$ is the
unique solution of \ref{eq:LTVsystem} defined on $\intcc{t_0,t_f}$ and
satisfying $\varphi(s,s,p,u) = p$.
The map $\varphi(t,s,\cdot,0)$, which is linear, is called the
\concept{transition matrix at $(t,s)$} of the system and is denoted
by $\phi(t,s)$.
The map
$\phi \colon \intcc{t_0,t_f} \times \intcc{t_0,t_f} \to \mathbb{R}^{n \times n}$
is of class $C^1$, and the identities
\[
\varphi(t,s,p,u)
=
\phi(t,s) p
+
\int_s^t
\phi(t,\tau) B(\tau) u(\tau)
d\tau,
\]
$\phi(s,s) = \id$, and
$\phi(t,s) \phi(s,T) = \phi(t,T)$
hold for all $s, t, T \in \intcc{t_0,t_f}$, all $p \in \mathbb{R}^n$,
and all integrable $u$; see, e.g.~\cite{Lukes82}.
Moreover, $D_1 \phi(t,s) = A(t) \phi(t,s)$ and
$D_2 \phi(s,t) = - \phi(s,t) A(t)$
hold for all $s, t \in \intcc{t_0,t_f}$,
where $D_i \phi$ denotes the partial derivative of
$\phi$ with respect to (w.r.t.) the $i$th argument. If $A$ is additionally of
class $C^k$, $k \ge 1$, then $\phi$ is of class $C^{k+1}$. Finally,
Gronwall's lemma implies
\begin{equation}
\label{e:ExpEstimates}
\| \phi(t,s) \| \le \e^{ |t-s| M }
\text{\ \ and\ \ }
\| \phi(t,s) - \id \| \le \e^{ |t-s| M } - 1
\end{equation}
for all $s,t \in \intcc{t_0,t_f}$, provided that
$\| A(t) \| \le M$ for all $t \in \intcc{t_0,t_f}$.

\subsection{Reachable Sets and Tubes}
\label{ss:ReachableSetsAndTubes}

Given non-empty, compact, convex subsets $X_0 \subseteq \mathbb{R}^n$
and $U \subseteq \mathbb{R}^m$, and $a, b, t \in \intcc{t_0,t_f}$
satisfying $a \le b$, the sets
\begin{align*}
\mathcal{R}(t)
&=
\Set{ \varphi(t, t_0, x_0, u) }{ x_0 \in X_0, u\in U^{\intcc{t_{0},t}}},\\
\mathcal{R}( \intcc{a,b} )
&=
\bigcup_{s \in \intcc{a,b}}\mathcal{R}(s)
\end{align*}
are the \concept{reachable set at time $t$} and the \concept{reachable
tube over the time interval $\intcc{a,b}$}, respectively, of the
system \ref{eq:LTVsystem}.
Both $\mathcal{R}(t)$ and $\mathcal{R}( \intcc{a,b} )$ are non-empty
and compact, and $\mathcal{R}(t)$ is additionally convex and is
conveniently written using a set-valued integral,
$
\mathcal{R}(t)
=
\phi(t,t_0) X_0
+ \int_{t_0}^{t} \phi(t,s) B(s) U ds.
$
See, e.g.~\cite{Hermes70}. Moreover, the well-known semi-group
property of reachable sets
\cite{Chernousko94}
yields the identity

\begin{equation}
\label{e:SemiGroupProperty}
\mathcal{R}(b)
=
\phi(b,a) \mathcal{R}(a)
+
\int_a^b \phi(b,s) B(s) U ds.
\end{equation}

\section{Problem Statement}
\label{s:ProblemStatement}

We consider the system \ref{eq:LTVsystem}, where both the initial
state $x(t_0) \in X_0$ and the input $u(t) \in U$ are subject to
uncertainty, represented by the set $X_0$ and $U$, respectively. We
assume the following.

\begin{enumerate}[label=(\boldmath$H_{\arabic*}$),leftmargin=0pt,itemindent=*,ref=($H_{\arabic*}$)]
\item
\label{i:ProblemStatement:First}
\label{i:ProblemStatement:Time}
$n \in \mathbb{N}$, $t_0, t_f \in \mathbb{R}$ and $t_0 < t_f$.
\item
\label{i:ProblemStatement:Uncertainty}
$X_0 \subseteq \mathbb{R}^n$ and $U \subseteq \mathbb{R}^m$
are non-empty, compact, and convex.
\item
\label{i:ProblemStatement:Matrices}
$A$ and $B$ are of class $C^1$, and
$\| A(t) \| \le M_A$,
$\|\dot{A}(t) \| \le M_{\dot{A}}$,
$\| B(t) \| \le M_B$,
and
$\|\dot{B}(t)\| \le M_{\dot{B}}$
for all $t \in \intcc{t_0,t_f}$,
where $M_A,M_{\dot{A}}, M_B, M_{\dot{B}} \in \mathbb{R}$ and
$M_A > 0$.
Here, $\dot{A}$ denotes the derivative of the map
$A \colon \intcc{t_0,t_f} \to \mathbb{R}^{n \times n}$, and similarly
for $\dot{B}$.
\item
\label{i:ProblemStatement:ApproxPHI}
Denote
$D = \Set{(t,s) \in \intcc{t_0,t_f} \times \intcc{t_0,t_f}}{t \ge s}$.
Then
\begin{align}
\label{i:ProblemStatement:ApproxPHI:LocalError}
\| \phi(t,s) - \widetilde{\phi}(t,s) \|
&\le
\theta (t-s)
\text{ for all $(t,s) \in D$},
\\
\label{i:ProblemStatement:ApproxPHI:LocalErrorOfOrder2}
\theta(h) &= O(h^2)
\text{ as $h \to 0$},
\end{align}
where
$\widetilde{\phi} \colon D \to \mathbb{R}^{n\times n}$
approximates the transition matrix $\phi$ of \ref{eq:LTVsystem} and
$\theta \colon \mathbb{R}_{+} \to \mathbb{R}_{+}$ is
monotonically increasing.
\label{i:ProblemStatement:Last}
\end{enumerate}

We note that \ref{i:ProblemStatement:ApproxPHI} is the requirement
that the approximation $\widetilde{\phi}$ of $\phi$ has consistency
order $1$
\cite[Def.~4.7]{DeuflhardBornemann94Vol2ENGLISH}.
Under assumptions
\ref{i:ProblemStatement:First}-\ref{i:ProblemStatement:Matrices},
this requirement is satisfied by the vast majority of numerical
methods to solve initial value problems. See,
e.g.~\cite[Example 4.8]{DeuflhardBornemann94Vol2ENGLISH},
as well as Lemma \ref{lem:SecondOrderTaylorsMethod} in the Appendix.

The problem data $t_0$, $t_f$, $A$, $B$, $X_0$ and $U$ are fixed
throughout the paper, and so are the constants $M_A$, $M_{\dot{A}}$,
$M_B$, and $M_{\dot{B}}$, as well as the functions
$\phi$, $\widetilde{\phi}$ and $\theta$ and the set $D$.
Throughout the paper, all that data is subject to the standing
hypotheses
\ref{i:ProblemStatement:First}-\ref{i:ProblemStatement:Last}.

\begin{problem}
\label{problem:definition}
Devise a convergent method that over-approximates
$\mathcal{R}(\intcc{t_0,t_f})$, in the sense that given the problem
data and a time discretization parameter $N$,
a superset $\widehat R_N$ of $\mathcal{R}(\intcc{t_0,t_f})$ is
obtained, satisfying $\widehat R_N \to \mathcal{R}(\intcc{t_0,t_f})$
in Hausdorff distance as $N \to \infty$.
\end{problem}

\section{Proposed method}
\label{s:MainResults}

In order to solve Problem \ref{problem:definition} for any given value
of the time discretization parameter $N$, we shall over-approximate
reachable sets $\mathcal{R}(t_i)$ and
reachable tubes $\mathcal{R}( \intcc{t_i, t_{i+1}} )$
of the control system \ref{eq:LTVsystem} for equidistant points of
time $t_i \in \intcc{t_0, t_f}$, $i \in \intcc{0;N}$.
The approximation will be convergent of first order
\cite{DeuflhardBornemann94Vol2ENGLISH},
meaning that
the Hausdorff distance between $\mathcal{R}(t_i)$ and its
approximation is of order $O(1/N)$, and similarly for
tubes.
The respective method of over-approximation,
presented in Section \ref{ss:MainResults:General},
applies to any uncertainty sets $X_0$ and $U$ satisfying Hypothesis
\ref{i:ProblemStatement:Uncertainty} and assumes the ability to
compute with compact convex sets in finite dimension.
Our algorithmic solution subsequently presented in Section
\ref{ss:MainResults:Zonotopes} applies to the case of zonotopic
uncertainties, uses only linear algebraic operations, and
involves an additional approximation step that yields
zonotopic over-approximations of reachable tubes
retaining first order accuracy.

\subsection{Over-approximation of Reachable Sets and Reachable Tubes}
\label{ss:MainResults:General}

We consider the system \ref{eq:LTVsystem} under our standing
hypotheses
\ref{i:ProblemStatement:First}-\ref{i:ProblemStatement:Last}.
Given a time discretization parameter $N \in \mathbb{N}$, we define
sequences
$( \Omega_i )_{i \in \intcc{0;N}}$ and
$( \Gamma_i )_{i \in \intcc{1;N}}$
of subsets $\Omega_i, \Gamma_i \subseteq \mathbb{R}^n$
by the following requirements for all $i \in \intcc{1;N}$.
\begin{subequations}
\label{ProposedMethod:General}
\begin{align}
\label{ProposedMethod:General:h}
h
&=
(t_f - t_0) / N
\text{\ \ and\ \ }
t_i = t_0 + i h,
\\
\label{ProposedMethod:General:omega0}
\Omega_0 &= X_0,
\\
\label{ProposedMethod:General:omegai}
\Omega_i
&=
\widetilde{\phi}(t_i,t_{i-1}) \Omega_{i-1}
+ h B(t_i) U
+ ( \alpha_h + \theta_h \| \Omega_{i-1} \| ) \mathbb{B},
\\
\label{ProposedMethod:General:gammai}
\Gamma_i
&=
\conv
\left(
\Omega_{i-1}
\cup
\left(
\Omega_i
+ ( \beta_h + \gamma_h \| \Omega_{i-1} \| ) \mathbb{B}
\right)
\right).
\end{align}
Here,
$\| \cdot \|$ denotes any norm
and the maps
$\alpha, \beta, \gamma \colon \mathbb{R}_{+} \to \mathbb{R}_{+}$ are
defined by
\begin{flalign}
\label{ProposedMethod:General:beta}
r(s)
&=
\exp(s M_A) -1 - s M_A,
&
\beta(s)
&=
s^2 M_{\dot{B}} \| U \|,
\\
\label{ProposedMethod:General:alphagamma}
\alpha(s)
&=
r(s) \| U \| \frac{ M_{\dot{B}} + M_{A} M_{B} }{ M_{A}^2 },
&
\gamma(s)
&=
r(s) \left( 1 + \frac{ M_{\dot{A}} }{ M_{A}^{2} } \right)
\negthinspace.
\end{flalign}
\end{subequations}
For convenience, here and in the sequel we often use
$\alpha_h$ in place of $\alpha(h)$, and similarly for $\beta$,
$\gamma$ and $\theta$.

By \ref{ProposedMethod:General:h}, we define an equidistant grid with
step size $h$, of points $t_0$, \dots, $t_N$, spanning the
time interval $\intcc{t_0,t_f}$.
The requirements
\ref{ProposedMethod:General:omega0}-\ref{ProposedMethod:General:omegai}
iteratively define sets $\Omega_i$, which are supposed to approximate
the reachable sets $\mathcal{R}(t_i)$, and in turn,
\ref{ProposedMethod:General:gammai} uses these approximations as
well as their inflated versions to define sets $\Gamma_i$, which are
supposed to approximate reachable tubes $\mathcal{R}(\intcc{t_{i-1},t_i})$.
As we shall show, due to our careful definition of the maps $\alpha$,
$\beta$ and $\gamma$ depending on the time-varying problem data,
both $\Omega_i$ and $\Gamma_i$ actually are over-approximations, with
approximation error of order $O(1/N)$.

We now set out to state formally and to prove what we have just
described in informal terms. In doing so, we shall use the superscript
$N$ to indicate that, e.g.~the sequence
$( \Omega^N_i )_{i \in \intcc{0;N}}$has been computed by our method
\ref{ProposedMethod:General} for a specific value of the time
discretization parameter, and similarly for $h$, $t_i$ and $\Gamma_i$.

\begin{proposition}[Reachable Sets]
\label{prop:OverApproximatingReachableSets}
For each $N \in \mathbb{N}$, let sequences
$(t^N_i)_{i \in \intcc{0;N}}$ and $(\Omega^N_i)_{i \in \intcc{0;N}}$
be defined by \ref{ProposedMethod:General:h}-\ref{ProposedMethod:General:omegai}
and \ref{ProposedMethod:General:beta}-\ref{ProposedMethod:General:alphagamma}.
\\
Then 
$\mathcal{R}(t^N_i) \subseteq \Omega^N_i$
for all $N \in \mathbb{N}$ and all $i \in \intcc{0;N}$, and
$d_H( \mathcal{R}(t^N_i), \Omega^N_i ) \le O(1/N)$
as $N \to \infty$, uniformly w.r.t.~$i$.
\end{proposition}

\noindent
For our proof, we need the following auxiliary results.

\begin{lemma}
\label{lem:ApproximatingSetValuedIntegral}
We have the estimate
$d_H( I(a,b), J(a,b) ) \le \alpha( b - a )$
whenever $t_0 \le a \le b \le t_f$, where
$I(a,b) = \int_a^b \phi(b,s) B(s) U ds$,
$J(a,b) = (b-a) B(b) U$, and
$\alpha$ is defined in \ref{ProposedMethod:General:alphagamma}.
\end{lemma}

\begin{proof}
Let $a, b \in \intcc{t_0,t_f}$, $a \le b$. The assumption
\ref{i:ProblemStatement:Uncertainty} on $U$ implies
$J(a,b) = \int_a^b B(b) U ds$, and using
Filippov's Lemma \cite{Hermes70}, we obtain
\begin{equation}
\label{lem:ApproximatingSetValuedIntegral:proof:1}
d_H ( I(a,b), J(a,b) )
\le
\| U \|
\int_a^b \| \phi(b,s) B(s) - B(b) \| ds.
\end{equation}
Next, using \ref{e:ExpEstimates} and the identity
\[
\phi(b,s) B(s) - B(b)
= \int_b^s \phi(b,z) ( \dot{B}(z) - A(z) B(z) ) dz
\]
for all $s \in \intcc{t_0,t_f}$, we see that the integrand in
\ref{lem:ApproximatingSetValuedIntegral:proof:1} is bounded by 
$
( M_{\dot{B}} + M_{A} M_{B} ) (\e^{(b-s)M_{A}} - 1) / M_A
$, which proves the lemma.
\end{proof}

\begin{lemma}
\label{lem:UniformBoundednessOfOverapproximations}
Let $a \in \mathbb{R}_{+}$, $b \in \mathbb{R}$,
$K \colon \mathbb{N} \to \mathbb{N}$, and for each
$N \in \mathbb{N}$, let $(x^N_i)_{i \in \intcc{0;K(N)}}$ be a sequence
in $\mathbb{R}_{+}$. Suppose that $K(N) = O(N^a)$,
$x^N_0 = O( N^{a+b} )$ and
$x^N_i \le ( 1 + O(N^{-a}) ) x^N_{i-1} + O( N^b )$ hold
as $N \to \infty$, uniformly w.r.t.~$i$.
\\
Then $x^N_i \le O( N^{a+b} )$ as $N \to \infty$, uniformly w.r.t.~$i$.
\end{lemma}

\begin{proof}
By our hypotheses, there exist maps
$p,q,r \colon \mathbb{N} \to \mathbb{R}_{+}$
satisfying
$p(N) = O(N^{a+b})$,
$q(N) = O(N^{-a})$ and
$r(N) = O(N^b)$
as $N \to \infty$, and
\begin{equation}
\label{lem:UniformBoundednessOfOverapproximations:proof:1}
x^N_0 \le p(N)
\text{\ \ and\ \ }
x^N_i \le (1 + q(N) ) x^N_{i-1} + r(N)
\end{equation}
for all sufficiently large $N \in \mathbb{N}$ and all
$i \in \intcc{1;K(N)}$.
Define $f(N,i) = (1 + q(N) )^i$ for all $N \in \mathbb{N}$ and all
$i \in \intcc{0;K(N)}$, to arrive at
$f(N,i) \le \exp( q(N) K(N) )$. Then, by our assumptions on $q$ and
$K$, the map $f$ is bounded. In view of
\ref{lem:UniformBoundednessOfOverapproximations:proof:1} and the
variation-of-constants formula we conclude that
$x^N_i \le O( N^{a+b} )$ as claimed.
\end{proof}

\begin{proof}[Proof of Proposition \ref{prop:OverApproximatingReachableSets}]
For the sake of simplicity, throughout this proof we drop
the superscript $N$ from our notation. Let $h$ be defined
by \ref{ProposedMethod:General:h}.

The first claim holds for $i = 0$ and all $N \in \mathbb{N}$ as
$\mathcal{R}(t_0) = X_0 = \Omega_0$.
Assume that $\mathcal{R}(t_i) \subseteq \Omega_i$ holds for some
$N \in \mathbb{N}$ and some $i \in \intco{0;N}$. Then, using the
identity \ref{e:SemiGroupProperty} and Lemma
\ref{lem:ApproximatingSetValuedIntegral}
as well as Lemma \ref{lem:HausdorffDistance}\ref{lem:HausdorffDistance:5},
we obtain
$
\mathcal{R}(t_{i+1})
\subseteq
\phi( t_{i+1}, t_i ) \Omega_i + h B(t_{i+1}) U + \alpha( h ) \mathbb{B}
$. Moreover,
$
\phi( t_{i+1}, t_i) \Omega_i
\subseteq
\widetilde{\phi}( t_{i+1}, t_i) \Omega_i + \theta(h) \| \Omega_i \|
$
by the estimate \ref{i:ProblemStatement:ApproxPHI:LocalError} and
Lemma \ref{lem:HausdorffDistance}\ref{lem:HausdorffDistance:3}\ref{lem:HausdorffDistance:5}, and so
$\mathcal{R}(t_{i+1}) \subseteq \Omega_{i+1}$.

To prove the second claim, we use the triangle inequality, assumption
\ref{i:ProblemStatement:Matrices} and estimates \ref{e:ExpEstimates},
\ref{i:ProblemStatement:ApproxPHI:LocalError} and
\ref{i:ProblemStatement:ApproxPHI:LocalErrorOfOrder2} to obtain the
bound
$\| \widetilde{\phi}(t_i, t_{i-1}) \| \le 1 + O(1/N)$ as $N \to \infty$,
uniformly w.r.t.~$i$. In turn,
\ref{i:ProblemStatement:ApproxPHI:LocalErrorOfOrder2},
\ref{ProposedMethod:General:omegai}, \ref{i:ProblemStatement:Uncertainty} and
the fact that $\alpha(s) = O(s^2)$ as $s \to 0$ together imply
$\| \Omega_i \| \le ( 1 + O(1/N) ) \| \Omega_{i-1} \| + O(1/N)$, and
so $\| \Omega_i \| \le O(1)$ as $N \to \infty$, uniformly w.r.t.~$i$,
by Lemma \ref{lem:UniformBoundednessOfOverapproximations}.
It follows that
$
\alpha(h) + \theta(h) \| \Omega_{i-1} \|
\le
O(1/N^2)
$
and
$
d_H( \widetilde{\phi}(t_i,t_{i-1}) \Omega_{i-1}, \phi(t_i,t_{i-1}) \mathcal{R}(t_{i-1}) )
\le
(1 + O(1/N)) e_{i-1} + O(1/N^2)
$,
where $e_i = d_H( \mathcal{R}(t_i), \Omega_i )$.
Moreover,
$
d_H(h B(t_i)U,\int_{t_{i-1}}^{t_i} \phi(t_i,s) B(s) U ds)
\le
O(1/N^2)
$
by Lemma \ref{lem:ApproximatingSetValuedIntegral}, and
so $e_i \le (1+O(1/N)) e_{i-1} + O(1/N^2)$ as $N \to \infty$,
uniformly w.r.t.~$i$. Then $e_i \le O(1/N)$ by Lemma
\ref{lem:UniformBoundednessOfOverapproximations}, as claimed.
\end{proof}

\noindent
The following theorem, and its corollary immediately obtained
using Lemma
\ref{lem:HausdorffDistance}\ref{lem:HausdorffDistance:Union},
provide a first solution to Problem \ref{problem:definition}.

\begin{theorem}[Reachable Tubes]
\label{th:OverApproximatingReachableTubes}
For each $N \in \mathbb{N}$, let sequences
$(t^N_i)_{i \in \intcc{0;N}}$ and $(\Gamma^N_i)_{i \in \intcc{1;N}}$
be defined by \ref{ProposedMethod:General}.
\\
Then $\mathcal{R}(\intcc{t_{i-1},t_i}) \subseteq \Gamma^N_i$
for all $N \in \mathbb{N}$ and all $i \in \intcc{1;N}$, and
$d_H(\mathcal{R}(\intcc{t_{i-1},t_i}), \Gamma^N_i) \le O(1/N)$
as $N \to \infty$, uniformly w.r.t.~$i$.
\end{theorem}

\begin{corollary}
\label{cor:th:OverApproximatingReachableTubes}
Under the hypotheses and in the notation of Theorem
\ref{th:OverApproximatingReachableTubes}, denote
$\widehat R_N = \cup_{i \in \intcc{1;N}} \Gamma^N_i$.
Then $\mathcal{R}(\intcc{t_0,t_f}) \subseteq \widehat R_N$ for all
$N \in \mathbb{N}$, and
$d_H( \mathcal{R}(\intcc{t_0,t_f}), \widehat R_N ) = O(1/N)$
as $N \to \infty$.
\end{corollary}

\noindent
Our proof of Theorem \ref{th:OverApproximatingReachableTubes} uses
the following auxiliary result.

\begin{lemma}
\label{lem:InterpolantMatrix}
Let $\gamma$ be defined by \ref{ProposedMethod:General:alphagamma}, and let
$a,b \in \intcc{t_0,t_f}$ with $a < b$.
Then
$\| \phi(t,a) - \psi(t,a,b) \| \le (t-a)(b-a)^{-1} \gamma(b-a)$
for all $t \in \intcc{a,b}$,
where
$
\psi(t,a,b)
=
\id + (\phi(b,a) - \id) (t-a) / (b-a)
$.
\end{lemma}

\begin{proof}
The claim is obvious for $t \in \{ a, b \}$, so we
suppose that $t_0 \le a < t < b \le t_f$. Then, by a change of
variable,
\[
\frac{t-a}{b-a}
\int_a^b
A(\tau(s)) \phi(\tau(s),a) ds
=
\int_a^t A(s) \phi(s,a) ds,
\]
where $\tau(s) = (t-a)(s-a)/(b-a) + a \in \intcc{a,s}$,
and so the difference $\phi(t,a) - \psi(t,a,b)$ can be written as
\begin{equation}
\label{lem:InterpolantMatrix:proof:e1}
\frac{t-a}{b-a}
\int_a^b
A(\tau(s)) \phi(\tau(s),a) - A(s)\phi(s,a)
ds.
\end{equation}
As the map $s \mapsto A(s) \phi(s,a)$ is smooth, the integrand in
\ref{lem:InterpolantMatrix:proof:e1} takes the form
$\int_s^{\tau(s)} ( \dot{A}(z) + A(z)^2 ) \phi(z,a) dz$. The claim then
follows from the estimate \ref{e:ExpEstimates} and assumption
\ref{i:ProblemStatement:Matrices}.
\end{proof}

We mention in passing that, in the time-invariant case of
\ref{eq:LTVsystem} with $B(t) = \id$, our estimates in Lemmas
\ref{lem:ApproximatingSetValuedIntegral} and
\ref{lem:InterpolantMatrix} reduce to those in
\cite[Lemma 2]{LeGuernicGirard09} and
\cite[p.~260, last inequ.]{LeGuernicGirard09}, respectively. Another
related but less precise estimate is given in
\cite[Lemma 1]{BotchkarevTripakis00}.
The mathematical tools we have used to treat the general
time-varying case are quite different from the ones used in
\cite{LeGuernicGirard09,BotchkarevTripakis00}.%

\begin{proof}[Proof of Theorem \ref{th:OverApproximatingReachableTubes}]
For the sake of simplicity, throughout this proof we drop
the superscript $N$ from our notation. Moreover, we do not
mention the domains $\mathbb{N}$, $\intcc{1;N}$ and
$\intcc{t_{i-1},t_i}$ of $N$, $i$ and $t$, and asymptotic estimates
are always meant to hold for $N \to \infty$,
uniformly w.r.t.~$i$ and $t$.
The map $\psi$ is defined in Lemma \ref{lem:InterpolantMatrix}, and
$h$ is defined in \ref{ProposedMethod:General:h}.

We claim that $\mathcal{R}(t) \subseteq E_{i,t}$ for all $N$, $i$ and
$t$, where
\begin{align*}
E_{t,i}
&=
\psi(t,t_{i-1},t_i) R(t_{i-1}) + (t - t_{i-1}) h^{-1} M_i,
\\
M_i
&=
h B(t_i) U + ( \alpha_h + \beta_h + \gamma_h \| \Omega_{i-1} \| ) \mathbb{B}.
\end{align*}
Indeed, using Lemmata \ref{lem:InterpolantMatrix} and
\ref{lem:HausdorffDistance}\ref{lem:HausdorffDistance:3},\ref{lem:HausdorffDistance:5},
the compactness of reachable tubes, and Proposition
\ref{prop:OverApproximatingReachableSets}, we see that
\begin{equation}
\label{th:OverApproximatingReachableTubes:proof:1}
\phi(t,t_{i-1}) R(t_{i-1})
\subseteq
\psi(t,t_{i-1},t_i) R(t_{i-1})
+
\frac{t-t_{i-1}}{h} \gamma_h \| \Omega_{i-1} \| \mathbb{B}.
\end{equation}
Moreover, we obviously have
$\| B(t_i) - B(t) \| \le h M_{\dot{B}}$, and in turn,
$
d_H( (t-t_{i-1}) B(t) U, (t-t_{i-1}) B(t_i) U )
\le
(t-t_{i-1}) \beta_h /h
$
by Lemma \ref{lem:HausdorffDistance}\ref{lem:HausdorffDistance:3}.
Then Lemmata \ref{lem:ApproximatingSetValuedIntegral} and
\ref{lem:HausdorffDistance}\ref{lem:HausdorffDistance:5},
the compactness of $U$, and the fact that $\alpha(s) / s$ is
monotonically increasing in $s$, imply
\begin{equation}
\label{th:OverApproximatingReachableTubes:proof:2}
\int_{t_{i-1}}^t \phi(t,s) B(s) U ds
\subseteq
\frac{t - t_{i-1}}{h}
(
h B(t_i) U + ( \alpha_h + \beta_h ) \mathbb{B}
)
\end{equation}
for all $N$, $i$ and $t$. Our claim then follows from the identity
\ref{e:SemiGroupProperty}. Moreover, the estimates from which
the inclusions \ref{th:OverApproximatingReachableTubes:proof:1} and
\ref{th:OverApproximatingReachableTubes:proof:2} have been obtained
also show that $d_H( R(t), E_{i,t} ) \le O(1/N^2)$.

Next observe that the set $E_{i,t}$ takes the form
$
\Set{ (1-\lambda) x + \lambda \phi(t_i,t_{i-1}) x + \lambda m}{%
x \in \mathcal{R}(t_{i-1}), m \in M_i}
$,
where $\lambda = (t - t_{i-1}) / h$, and so
$E_{i,t} \subseteq F_{i,t}$ for all $N$, $i$ and $t$, where
$F_{i,t}$ is defined to be the set
\[
\Set{ (1-\lambda) x + \lambda \phi(t_i,t_{i-1}) y + \lambda m}{%
x,y \in \mathcal{R}(t_{i-1}), m \in M_i}.
\]
Moreover, if $z \in F_{i,t}$, then there exist
$x,y \in \mathcal{R}(t_{i-1})$ and $m \in M_i$ satisfying
$z = (1-\lambda) x + \lambda \phi(t_i,t_{i-1}) y + \lambda m$. We
define
$x' = (1-\lambda)x + \lambda y \in \mathcal{R}(t_{i-1})$ and
$z' = (1-\lambda)x' + \lambda \phi(t_i,t_{i-1})x' + \lambda m$
to obtain
$
z - z'
=
\lambda (1-\lambda) (\phi(t_i,t_{i-1}) - \id) (y-x)
$.
The estimate \ref{e:ExpEstimates} then implies
$
\| z - z' \|
\le
( \exp(h M_A) - 1 ) \| \Omega_{i-1} \| / 2
$,
and as $\| \Omega_i \| \le O(1)$ by Proposition
\ref{prop:OverApproximatingReachableSets} and the compactness of
reachable tubes, we arrive at $d_H( E_{i,t}, F_{i,t} ) \le O(1/N)$.

So far, we have shown that $\mathcal{R}(t) \subseteq F_{i,t}$ for all
$N$, $i$ and $t$, and that
$d_H( \mathcal{R}(t), F_{i,t} ) \le O(1/N)$. It follows that
$
\mathcal{R}( \intcc{ t_{i-1}, t_i } )
\subseteq
\cup_{t \in \intcc{ t_{i-1}, t_i } } F_{i,t}
$,
and by Lemma
\ref{lem:HausdorffDistance}\ref{lem:HausdorffDistance:Union}, the
Hausdorff distance between the two sets does not exceed $O(1/N)$.
Next observe that
\[
\cup_{t \in \intcc{ t_{i-1}, t_i } } F_{i,t}
=
\conv
\left(
\mathcal{R}( t_{i-1} )
\cup
( \phi(t_i,t_{i-1}) \mathcal{R}( t_{i-1} ) + M_i )
\right)
\]
by Lemma
\ref{lem:OperationsOnConvexSets}\ref{lem:OperationsOnConvexSets:CHOfUnionOfConvexSets},
and that
$
\phi(t_i,t_{i-1}) \mathcal{R}( t_{i-1} )
\subseteq
\widetilde{\phi} (t_i,t_{i-1}) \Omega_{i-1}
+
\theta(h) \| \Omega_{i-1} \| \mathbb{B}
$
by Lemma
\ref{lem:HausdorffDistance}\ref{lem:HausdorffDistance:3},\ref{lem:HausdorffDistance:5},
the estimate \ref{i:ProblemStatement:ApproxPHI:LocalError}, and
Proposition \ref{prop:OverApproximatingReachableSets}. Thus,
$
\cup_{t \in \intcc{ t_{i-1}, t_i } } F_{i,t}
\subseteq
\Gamma_i
$
for all $N$ and $i$, and the aforementioned results also show that the
distance of the two sets does not exceed $O(1/N)$, which completes
our proof.
\end{proof}

So far, we have demonstrated that our method
\ref{ProposedMethod:General} yields over-approximations of reachable
sets and tubes, for any uncertainty sets $X_0$ and $U$ satisfying
Hypothesis \ref{i:ProblemStatement:Uncertainty}, assuming the ability
to compute with compact convex sets in finite dimension.
By suitably representing these sets and the set operations in
\ref{ProposedMethod:General}, thereby possibly specializing to a
subclass of sets, the method can be implemented on a computer.
See e.g.~\cite{LeGuernic09,AlthoffFrehse16} for a
discussion of the merits of several classes of sets and their
representations in reachability analysis.

\subsection{Zonotopic Over-approximation}
\label{ss:MainResults:Zonotopes}

In this section, we present a variant of our method
\ref{ProposedMethod:General} for the class of zonotopes, i.e.,
for sets of the form
\begin{equation}
\label{def:Zonotopes:e}
\cZonotope(c,G)
=
c + G \intcc{-1,1}^q
\end{equation}
for some $c \in \mathbb{R}^n$, $G \in \mathbb{R}^{n \times q}$, and
$q \in \mathbb{Z}_{+}$, where $c$ is the \concept{center} and the
columns of $G$ are the \concept{generators} of $\cZonotope(c,G)$.
In particular, we assume that the uncertainty of the system
\ref{eq:LTVsystem} is given as zonotopes,
\begin{equation}
\label{e:ZonotopicUncertainty}
X_0
=
\cZonotope( a, E )
\text{\ \ and\ \ }
U
=
\cZonotope( c, G ),
\end{equation}
where $a \in \mathbb{R}^n$,
$c \in \mathbb{R}^m$,
$E \in \mathbb{R}^{n \times p}$,
$G \in \mathbb{R}^{m \times q}$, and
$p,q \in \mathbb{Z}_{+}$.

A problem with zonotopic implementations of
\ref{ProposedMethod:General} is that zonotopes are not closed under
convex hulls, and so the sets $\Gamma_i$ defined in
\ref{ProposedMethod:General:gammai} are not, in general, zonotopes.
We here follow an idea by \person{Girard} \cite{Girard05} and
replace $\Gamma_i$ by a zonotope obtained using the enclosure operator
$
\Enclosure
\colon
\left( \mathbb{R}^n \times \mathbb{R}^{n \times p} \right)^2
\to
\mathbb{R}^n \times \mathbb{R}^{n \times (2 p + 1)}
$
given by
\begin{equation}
\label{def:ZonotopeEnclosure:e}
\Enclosure((b,F),(c,G))
=
\left(
\frac{b+c}{2},\left(\frac{F+G}{2},\frac{b-c}{2},\frac{F-G}{2}\right)
\right)
\end{equation}
for all $b,c \in \mathbb{R}^n$, $F, G \in \mathbb{R}^{n \times p}$ and
$p \in \mathbb{Z}_{+}$.
Specifically, we propose the following variant of our method
\ref{ProposedMethod:General} for the case of zonotopic uncertainties 
\ref{e:ZonotopicUncertainty}.
Given a time discretization parameter $N \in \mathbb{N}$, we shall
compute sequences
$( b_i )_{i \in \intcc{0;N}}$,
$( F_i )_{i \in \intcc{0;N}}$,
$( d_i )_{i \in \intcc{1;N}}$ and
$( H_i )_{i \in \intcc{1;N}}$
satisfying the following conditions for all $i \in \intcc{1;N}$.
\begin{subequations}
\label{ProposedMethod:Zonotopes}
\begin{align}
\label{ProposedMethod:Zonotopes:b0F0}
b_0 &= a
\text{\ \ and\ \ }
F_0 = E,
\\
\label{ProposedMethod:Zonotopes:m}
m_{i-1}
&=
\| (b_{i-1}, F_{i-1} ) \|_{\infty}
\text{\ \ and\ \ }
K_i = h B(t_i) G,
\\
\label{ProposedMethod:Zonotopes:bi}
b_i
&=
\widetilde{\phi}(t_i,t_{i-1}) b_{i-1} + h B(t_i) c
\\
\label{ProposedMethod:Zonotopes:Fi}
F_i
&=
\left(
\widetilde{\phi}(t_i,t_{i-1}) F_{i-1},
K_i,
( \alpha_h + \theta_h m_{i-1} ) \id
\right),
\\
\label{ProposedMethod:Zonotopes:diJi}
(d_i,J_i)
&=
\Enclosure
\left(
( b_{i-1}, F_{i-1} ),
( b_i, \widetilde{\phi}(t_i,t_{i-1}) F_{i-1} )
\right),
\\
\label{ProposedMethod:Zonotopes:Hi}
H_i
&=
\left(
J_i,
K_i,
( \alpha_h + \beta_h + ( \gamma_h + \theta_h ) m_{i-1} ) \id
\right),
\end{align}
\end{subequations}
where $h$, $t_i$, $\alpha$, $\beta$ and $\gamma$ are given by
\ref{ProposedMethod:General:h}, \ref{ProposedMethod:General:beta} and
\ref{ProposedMethod:General:alphagamma} and the norm $\| \cdot \|$ in
\ref{ProposedMethod:General:beta} and
\ref{ProposedMethod:General:alphagamma} is the maximum norm.

We note that the norm in \ref{ProposedMethod:Zonotopes:m} is
straightforward to compute. See Lemma \ref{lem:Zonotopes}. Moreover,
\ref{ProposedMethod:Zonotopes:b0F0}-\ref{ProposedMethod:Zonotopes:Fi}
is a straightforward implementation of the set operations in
\ref{ProposedMethod:General:omega0}-\ref{ProposedMethod:General:omegai}
into linear algebraic operations on centers and generators, and using
induction it easily follows that
\begin{equation}
\label{e:ZonotopicReachableSetImplementation}
\Omega_i = \cZonotope( b_i, F_i )
\text{\ \ for all $i \in \intcc{0;N}$},
\end{equation}
provided that the norm $\| \cdot \|$ in \ref{ProposedMethod:General}
is the maximum norm and $\mathbb{B}$ is the respective closed unit
ball.
Thus, by Proposition \ref{prop:OverApproximatingReachableSets}, the
pairs $( b_i, F_i )$ produced by algorithm
\ref{ProposedMethod:Zonotopes} represent zonotopic
over-approximations of reachable sets $\mathcal{R}(t_i)$ with first
order approximation error.

The case of reachable tubes is more involved and is the subject of
Theorem \ref{th:ZonotopicOverApproximationofRT} and its Corollary
\ref{cor:th:ZonotopicOverApproximationofRT} below. We shall
demonstrate that the pairs $(d_i,H_i)$ produced by the algorithm
\ref{ProposedMethod:Zonotopes} represent zonotopes
$\cZonotope(d_i,H_i)$ over-approximating the sets $\Gamma_i$ defined
in \ref{ProposedMethod:General:gammai}. Then, by Theorem
\ref{th:OverApproximatingReachableTubes}, these zonotopes over-approximate
the reachable tubes $\mathcal{R}( \intcc{t_{i-1}, t_i})$, and we shall
also show that first order convergence is retained. This way, we
obtain a solution to Problem \ref{problem:definition} which applies in
the case that the uncertainty of the system \ref{eq:LTVsystem} is
given as zonotopes, and, in contrast to the more general algorithm
\ref{ProposedMethod:General}, this solution can be directly
implemented on a computer.
As before, we shall use the superscript $N$ to indicate that,
e.g.~the sequence $( F^N_i )_{i \in \intcc{0;N}}$ has been
computed by our method \ref{ProposedMethod:Zonotopes} for a specific
value of the time discretization parameter, and similarly for $b_i$, $d_i$
and $H_i$.

\begin{theorem}[Zonotopic Over-approximation of Reachable Tubes]
\label{th:ZonotopicOverApproximationofRT}
Assume \ref{e:ZonotopicUncertainty}, and
for each $N \in \mathbb{N}$, let sequences
$(t^N_i)_{i \in \intcc{0;N}}$,
$(d^N_i)_{i \in \intcc{1;N}}$
and $(H^N_i)_{i \in \intcc{1;N}}$
be defined by
\ref{ProposedMethod:General:h}, \ref{ProposedMethod:General:beta},
\ref{ProposedMethod:General:alphagamma} and
\ref{ProposedMethod:Zonotopes}, where the norm $\| \cdot \|$ in
\ref{ProposedMethod:General:beta} and
\ref{ProposedMethod:General:alphagamma} is the maximum norm,
and denote $\Lambda^N_i = \cZonotope( d^N_i, H^N_i )$.
\\
Then $\mathcal{R}(\intcc{t_{i-1},t_i}) \subseteq \Lambda^N_i$
for all $N \in \mathbb{N}$ and all $i \in \intcc{1;N}$, and
$d_H(\mathcal{R}(\intcc{t_{i-1},t_i}), \Lambda^N_i) \le O(1/N)$
as $N \to \infty$, uniformly w.r.t.~$i$.
\end{theorem}

\begin{corollary}
\label{cor:th:ZonotopicOverApproximationofRT}
Under the hypotheses and in the notation of Theorem
\ref{th:ZonotopicOverApproximationofRT}, denote
$\widehat R_N = \cup_{i \in \intcc{1;N}} \Lambda^N_i$.
Then $\mathcal{R}(\intcc{t_0,t_f}) \subseteq \widehat R_N$ for all
$N \in \mathbb{N}$, and
$d_H( \mathcal{R}(\intcc{t_0,t_f}), \widehat R_N ) = O(1/N)$
as $N \to \infty$.
\end{corollary}

\noindent
Our proof of Theorem \ref{th:ZonotopicOverApproximationofRT} uses
the following auxiliary result.

\begin{lemma}
\label{lem:OverApproximatingCH}
Let $\Omega, \Gamma, W \subseteq \mathbb{R}^n$ be non-empty, compact
and convex, and suppose that
$0 \in W$.
Then
\begin{equation}
\label{e:lem:OverApproximatingCH}
\conv( \Omega \cup( \Gamma + W ) )
\subseteq
W + \conv( \Omega \cup \Gamma ),
\end{equation}
and
the Hausdorff distance between the two sets does not exceed $\| W \|$.
\end{lemma}

\begin{proof}
Let $r \in \conv( \Omega \cup( \Gamma + W ) )$. Then, by Lemma
\ref{lem:OperationsOnConvexSets}\ref{lem:OperationsOnConvexSets:CHOfUnionOfConvexSets},
there exist $\lambda \in \intcc{0,1}$, $x \in \Omega$,
$y \in \Gamma$ and $z \in W$ such that
\[
r
= \lambda x + ( 1 - \lambda )( y + z )
= \lambda x + ( 1 - \lambda ) y + ( 1 - \lambda ) z.
\]
Notice that $( 1 - \lambda ) z \in W$ as $0,z \in W$.
Hence, $r \in W + \conv( \Omega \cup \Gamma )$ which implies
\ref{e:lem:OverApproximatingCH}.
Let $s \in W + \conv( \Omega \cup \Gamma )$, then there exist
$\lambda \in \intcc{0,1}$, $x \in \Omega$, $y \in \Gamma$, $z \in W$
such that 
$
s = \lambda x+ ( 1 - \lambda ) y + z
$.
Define
$
t
=
\lambda x + ( 1 - \lambda ) ( y + z )
\in
\conv( \Omega \cup( \Gamma + W ) )
$.
Then we have $s - t = \lambda z$, and so
$
\| s - t \| \le \| W \|
$,
which proves the bound.
\end{proof}

\begin{proof}[Proof of Theorem \ref{th:ZonotopicOverApproximationofRT}]
For each $N \in \mathbb{N}$, let $h^N$ and sequences
$(\Omega^N_i)_{i \in \intcc{0;N}}$,
$(\Gamma^N_i)_{i \in \intcc{1;N}}$,
$(b^N_i)_{i \in \intcc{0;N}}$,
$(F^N_i)_{i \in \intcc{0;N}}$,
$(J^N_i)_{i \in \intcc{1;N}}$,
$(K^N_i)_{i \in \intcc{1;N}}$ and
$(m_i^N)_{i \in \intco{0;N}}$
be defined by
\ref{ProposedMethod:General} and \ref{ProposedMethod:Zonotopes}.
In the sequel, we drop
the superscript $N$ from our notation, and often we do
not mention the domains $\mathbb{N}$ and $\intcc{1;N}$
of $N$ and $i$. Everything is w.r.t.~the maximum norm, here
denoted by $\| \cdot \|$. This applies, in particular, to the norm
$\| \cdot \|$ and to the unit ball $\mathbb{B}$ in
\ref{ProposedMethod:General}.

We claim that $\Gamma_i \subseteq \Lambda_i$ for all $N$ and $i$, and
that $d_H( \Gamma_i, \Lambda_i ) \le O(1/N)$,
where asymptotic estimates are always meant to hold
for $N \to \infty$, uniformly w.r.t.~$i$.
The theorem then follows from an application of Theorem
\ref{th:OverApproximatingReachableTubes}.

To prove the claim, let $N \in \mathbb{N}$ and $i \in \intcc{1;N}$,
and denote
$P = \Omega_{i-1}$,
$L = \widetilde{\phi}(t_i,t_{i-1})$,
$w = h B(t_i) c$,
$M = L P + w$,
and
$
W
=
h B(t_i)( U - c )
+
( \alpha_h + \beta_h + ( \gamma_h + \theta_h ) \| P \| ) \mathbb{B}
$.
Then $\Gamma_i = \conv( P \cup( M + W ) )$ by
\ref{ProposedMethod:General:omegai} and
\ref{ProposedMethod:General:gammai}, and so
$
\Gamma_i
\subseteq
W + \conv( P \cup (L P + w ) )
$
by Lemma \ref{lem:OverApproximatingCH}, and in turn, Lemma
\ref{lem:Zonotopes}\ref{lem:Zonotopes:Enclosure},
\ref{ProposedMethod:Zonotopes:bi} and
\ref{ProposedMethod:Zonotopes:diJi} imply
$
\Gamma_i
\subseteq
W + \cZonotope( d_i, J_i )$.
Then $W + \cZonotope( d_i, J_i ) = \Lambda_i$ by
\ref{ProposedMethod:Zonotopes:m}, \ref{ProposedMethod:Zonotopes:Hi}
and Lemma \ref{lem:Zonotopes}\ref{lem:Zonotopes:Addition},
which proves the first part of our claim. From Lemmata
\ref{lem:OverApproximatingCH} and
\ref{lem:Zonotopes}\ref{lem:Zonotopes:Enclosure}
we additionally obtain the bound
$
d_H( \Gamma_i, \Lambda_i )
\le
\| W \| + \| L - \id \| \| F_{i-1} \|
$.
Next, Lemma \ref{lem:Zonotopes}\ref{lem:Zonotopes:Norm} shows that
$\| F_{i-1} \| \le \| \Omega_{i-1} \|$, and the triangular inequality,
Proposition \ref{prop:OverApproximatingReachableSets}, and the
estimate \ref{e:ExpEstimates} yield $\| L - \id \| \le O(1/N)$.
Finally, by the boundedness of $B$ from assumption
\ref{i:ProblemStatement:Matrices}, and by the fact that
$\| \Omega_i \| \le O(1)$ by Proposition
\ref{prop:OverApproximatingReachableSets} and the compactness of
reachable tubes, we obtain $\| W \| = O(1/N)$, which implies
$d_H( \Gamma_i, \Lambda_i ) = O(1/N)$ as claimed.
\end{proof}

\begin{figure}[!t]
\begin{center}
\includegraphics[width=.49\linewidth,keepaspectratio=true]{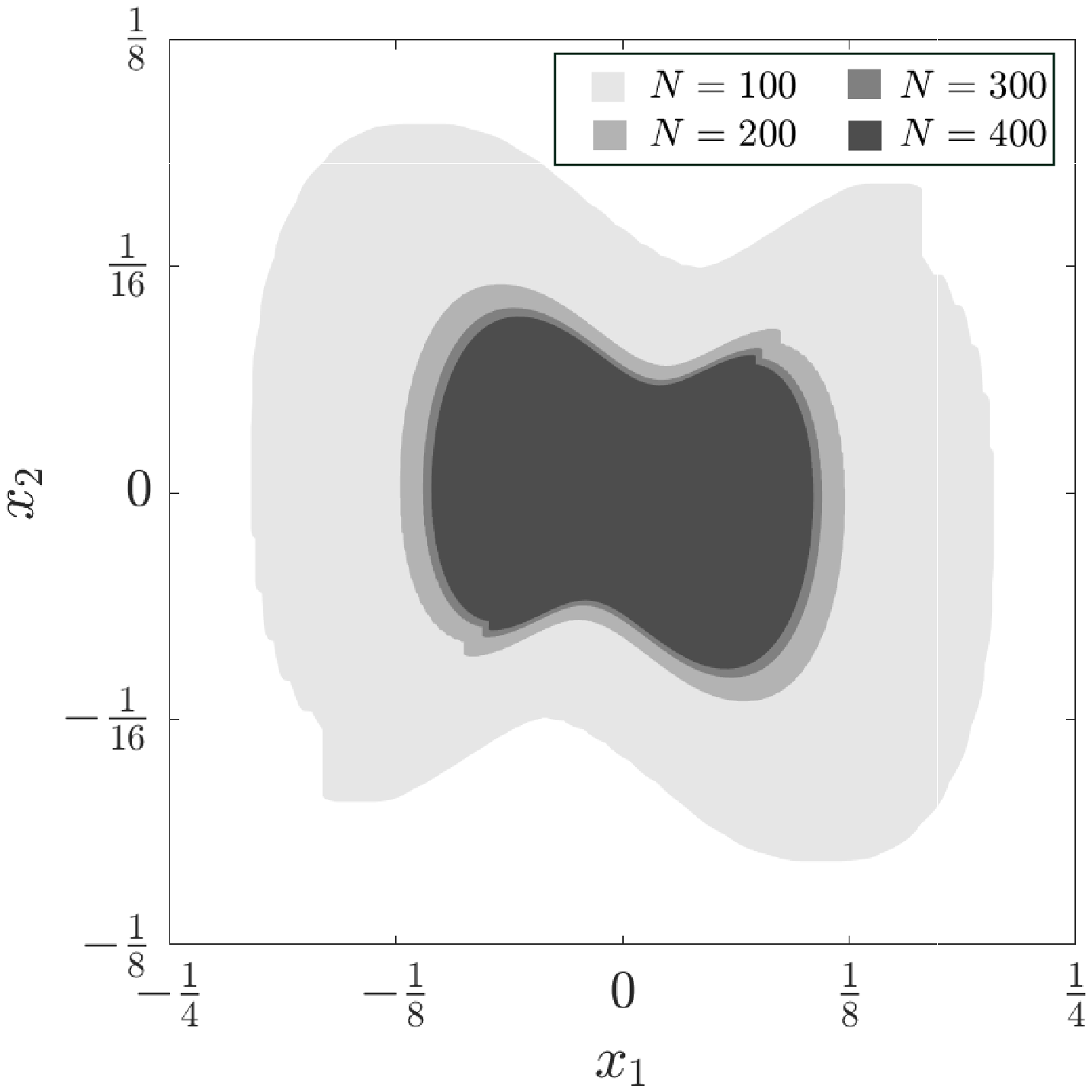}
\includegraphics[width=.49\linewidth,keepaspectratio=true]{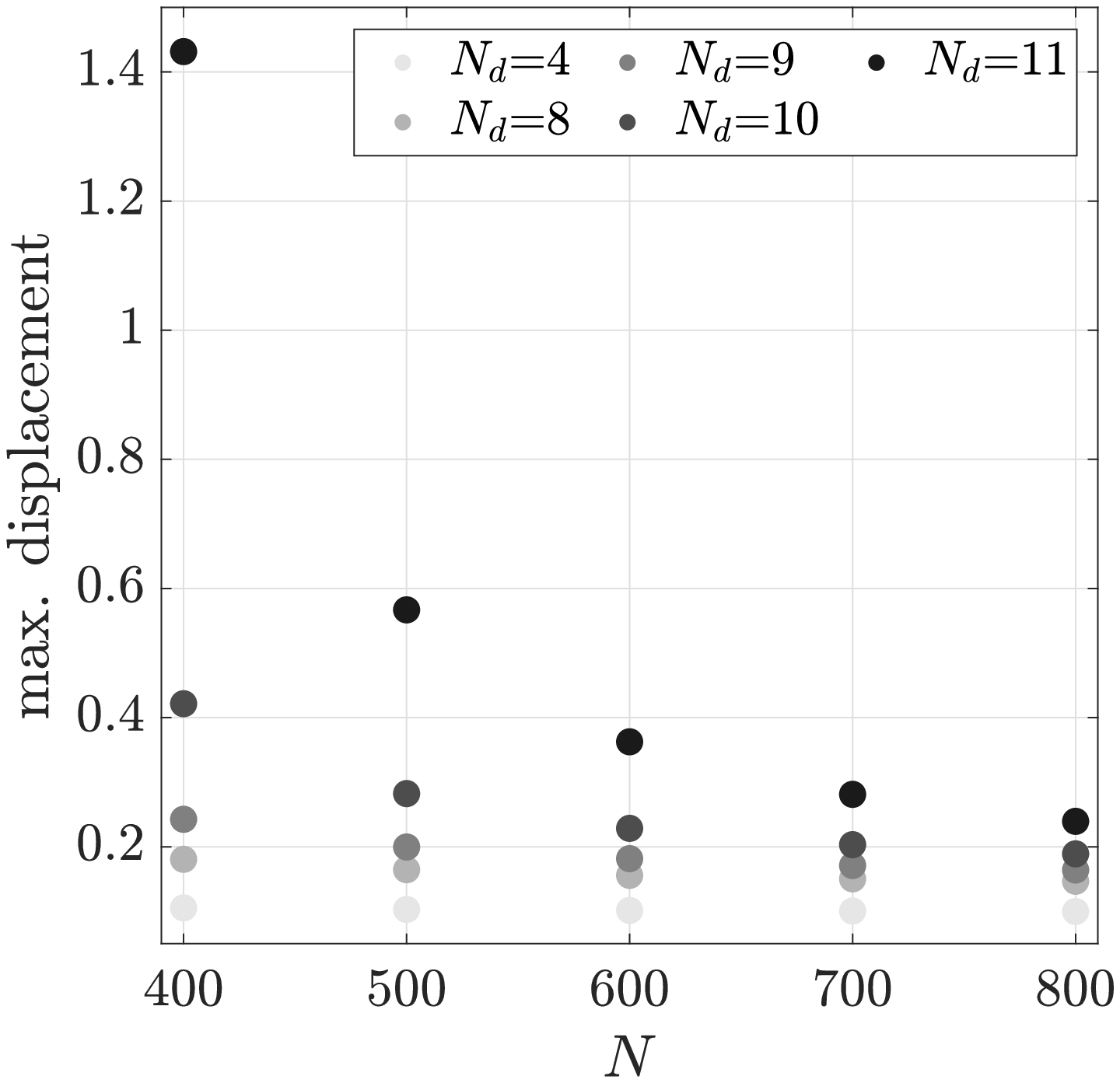}
\end{center}
\caption{\label{fig:ExampleReachableTubeOverApprox}
Zonotopic over-approximations computed by the proposed method for the
example in Section \ref{s:examples}, for selected values of the
discretization parameters $N$ and $N_d$:
Reachable tube $\mathcal{R}(\intcc{t_0,t_f})$ for $N_d=4$ (left), and
maximum bridge displacement, upon all nodal points, obtained from
reachable tube $\mathcal{R}(\intcc{t_0,t_f})$, for
$N_d \in \{ 4,8,9,10,11\}$ (right).}
\end{figure}

To close this section, we discuss the complexity of the proposed
method.
It is easily seen that the memory requirement of algorithm
\ref{ProposedMethod:Zonotopes} is determined by the need
to store the computed zonotopic over-approximation.
The zonotope $\Lambda_i^N$, $i \in \intcc{1;N}$, obtained in Theorem
\ref{th:ZonotopicOverApproximationofRT}, has
$2 p + 1 + ( 2 i - 1 )( q + n )$ generators, and consequently, the
over-approximation $\widehat R_N$ obtained in Corollary
\ref{cor:th:ZonotopicOverApproximationofRT} consists of $N$ zonotopes
in $\mathbb{R}^n$ with a total of $( q + n ) N^2 + ( 2 p + 1 ) N$
generators. In particular, the memory required by the zonotopic
over-approximation, and in turn, the memory required by our method,
is of order $O( n^2 N^2 )$ as one of the variables $n$ or $N$ tends to
infinity and the other one is fixed, where we have assumed both
$p = O(n)$ and $q = O(n)$. Regarding the run time of algorithm
\ref{ProposedMethod:Zonotopes}, we additionally assume $m = O(n)$ and
consider only arithmetic operations. Then the computational effort is
dominated by the multiplication of an $n \times n$ matrix by an
$n \times ( p + ( i - 1 ) ( q + n ) )$ matrix in step
\ref{ProposedMethod:Zonotopes:Fi}. Hence, the run time of algorithm
\ref{ProposedMethod:Zonotopes} is of order $O( n^3 N^2 )$ as one of
the variables $n$ or $N$ tends to infinity and the other one is fixed.

\section{Numerical Example}
\label{s:examples}

\begin{figure}[t]
\begin{center}
\includegraphics[width=.49\linewidth,keepaspectratio=true]{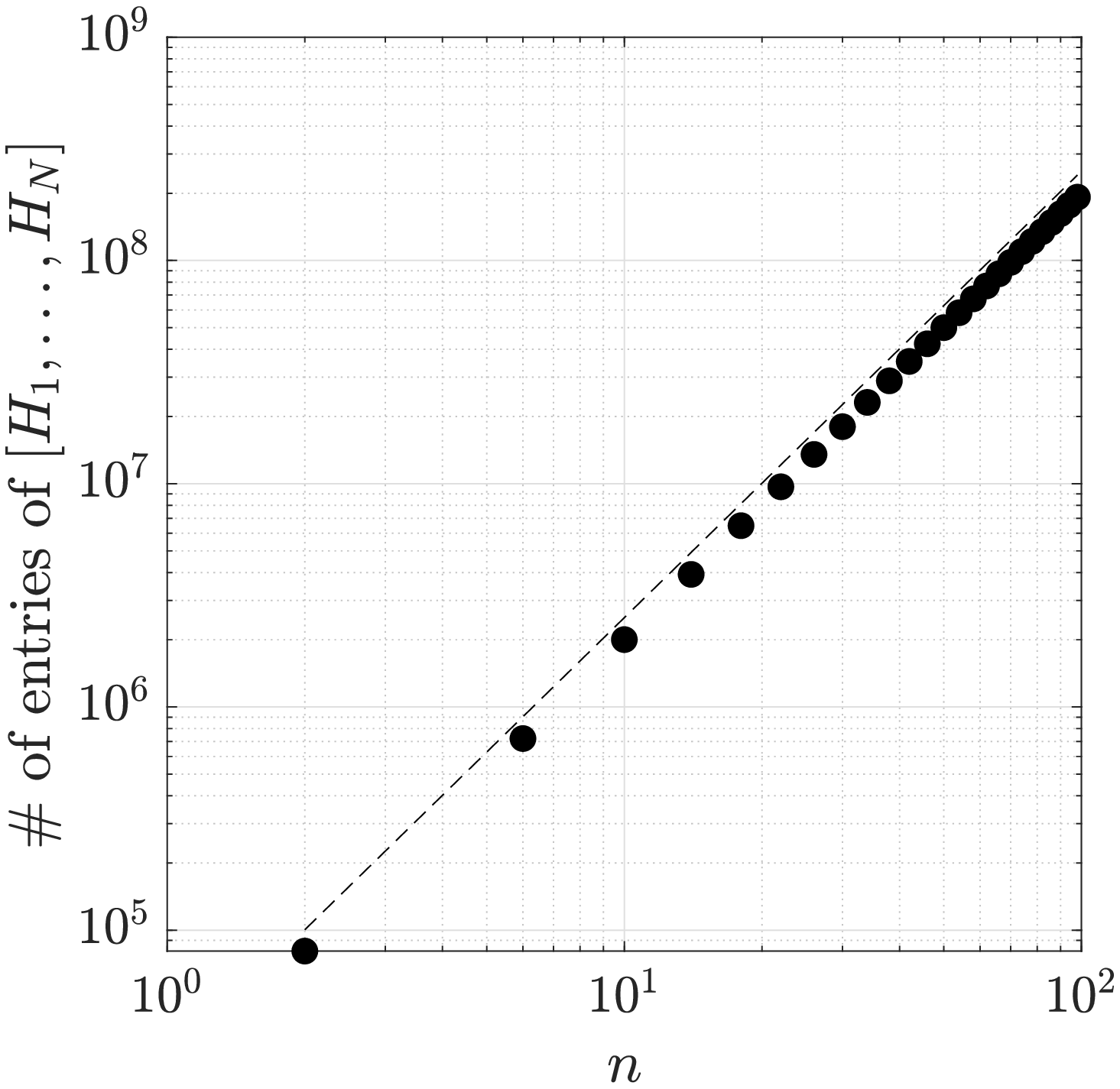}
\includegraphics[width=.49\linewidth,keepaspectratio=true]{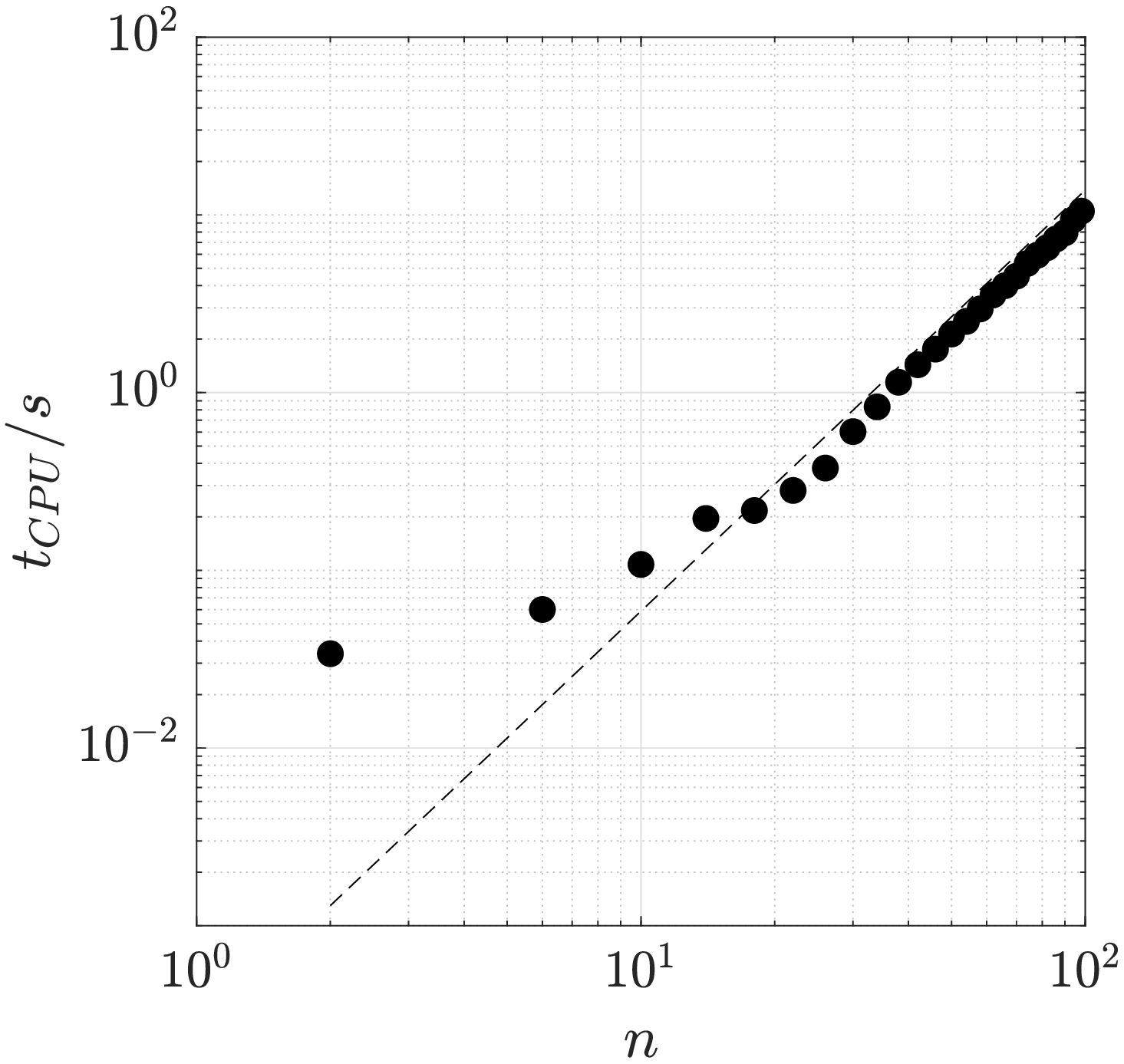}
\end{center}
\caption{\label{fig:ExamplePerformance}
Performance of the proposed method on the example in Section
\ref{s:examples} for the time discretization parameter $N = 100$,
depending on the dimension $n$ of the state space of the system
\ref{eq:LTVsystem}:
Total number of components of the generators of the computed zonotopic
over-approximation (left), and run time (right).}
\end{figure}

In this section, we demonstrate the performance of the proposed method
on a comprehensive numerical example. Specifically, we consider
several instances of a reduced order model, obtained by finite
difference approximation, of an infinite dimensional
pedestrians-footbridge system.

\subsection{Footbridge Model}
\label{ss:examples:FootbridgeModel}

Analyzing the dynamic response of footbridges is crucial for their
structural integrity and the safety of pedestrians
\cite{PiccardoTubino08,Gazzola15}. In this example, we consider a
continuum model of a pedestrians-footbridge system. The pedestrians'
stroll on the footbridge generates dynamic load which triggers
deformation of the footbridge. The lateral displacement of the
footbridge is given by the function
$q\colon \mathcal{I}\times\intcc{0,L}\rightarrow \mathbb{R}$, where
$L$ denotes the length of the bridge, and
$\mathcal{I}=\intcc{t_0,t_f}$ is the compact time interval under
consideration. Let $\mathcal{S}=\mathcal{I}\times \intcc{0,L}$. The
function $q$ satisfies
(see,
e.g.~\cite[eq.~29]{PiccardoTubino08}%
\cite[eqs.~2.91, 2.92]{Gazzola15}%
\cite[eq.~2]{BodgiErlicherArgoul07})
\begin{subequations}
\label{eq:FootbridgeModel}
\begin{align}
\label{eq:PDEmodel}
m D_{1}^{2}q&+EI D_{2}^{4}q+cD_{1}q=f_{p}(t,y),~(t,y)\in \mathcal{S},
\\
\label{eq:FootbridgeModelPDE}
f_{p}(t,y)&=f_{0} \cos (\omega t)q(t,y)+w(t,y),~(t,y)\in \mathcal{S},
\\
\label{eq:FootbridgeModelUncertainty}
\abs{w(t,y)}&\leq \bar{w}, (t,y)\in \mathcal{S},
\\
\label{eq:FootbridgeModelBC}
q(t,0)&=q(t,L)=D_{2}^{2}q(t,0)=D_{2}^{2}q(t,L)=0,~t\in \mathcal{I},
\\
\label{eq:FootbridgeModelIC}
q(t_0,y)&=D_{1}q(t_0,y)=0,~y\in \intcc{0,L},
\end{align}
\end{subequations}
where $D_i^k q$ denotes the $k$th order partial derivative of the map
$q$ w.r.t.~its $i$th argument, $D_1 \defas D_1^1$, $m$ denotes the
mass per unit length, $EI$ is the bending stiffness, $c$ is a damping
coefficient, $f_{p}(t,y)$ is the load per unit length, $f_0$ and
$\omega$ are model parameters, $w(t,y)$ is a bounded uncertain term,
and $\bar{w}\in \mathbb{R}_{+}$ is the bound on the uncertainty.

\subsection{Reduced Order Model}

Now, we deduce a reduced order model from system
\ref{eq:FootbridgeModel} by means of finite difference. Let $N_d$ be
a spatial discretization parameter with
$N_d \ge 4$, and
$h_y\defas L/N_d$,
$y_i \defas i h_y$, $i \in \intcc{0;N_d}$.
By replacing
the $4$th order spatial derivative in \ref{eq:PDEmodel} with second
order centered difference (see, e.g.~\cite{Fornberg88}), and the
second order spacial derivatives in \ref{eq:FootbridgeModelBC} with
first order forward and backward differences, respectively, and
considering the homogeneous boundary and initial conditions in
\ref{eq:FootbridgeModelBC} and \ref{eq:FootbridgeModelIC}, we obtain
the approximating model
$
m\ddot{z}+ K(t)z+c\dot{z}=v(t),
$
where $z(0)=\dot{z}(0)=0$, $v(t)\in \intcc{-\bar{w},\bar{w}}^{N_{d}-3}$, $t\in \intcc{t_0,t_f}$, and $K(\cdot)$ is obtained from the
finite difference approximation. By setting $x = (z, \dot z)$, we
arrive at a problem for system \ref{eq:LTVsystem}, where
\begin{subequations}
\label{eq:PDEmodelReducedSystem}
\begin{align}
\label{eq:PDEmodelReducedMatrixAandInputU}
A(t)&=\left(
\begin{array}{cc}
0& \id\\
-\frac{1}{m}K(t)& -\frac{c}{m}\id
\end{array}
\right)
\text{\ \ and\ \ }
B = \left( \begin{array}{c} 0\\ \id \end{array} \right),
\\
\label{eq:PDEmodelReducedInputSetAndX0} 
U&= \intcc{-\bar{w}/m,\bar{w}/m}^{N_{d}-3},~X_{0}=\{0\}.
\end{align}
\end{subequations} 
The dimension of the system \ref{eq:LTVsystem} is $n=2(N_{d}-3)$.
We note that system \ref{eq:PDEmodelReducedSystem}, for the case
$N_{d}=4$, corresponds to a damped and perturbed version of the
well-known Mathieu equation which models various physical phenomena
and engineering systems; see, e.g.~\cite{FossenNijmeijer11}.
We also note that the approach followed above to obtain the reduced
order linear time-varying (LTV) problem \ref{eq:PDEmodelReducedSystem}
from \ref{eq:FootbridgeModel} has been followed in the literature to
construct benchmark problems for reachability analysis of linear
time-invariant (LTI) systems \cite{TranNguyenJohnson16}.
In this example, we set
$L = 10$,
$m = 2$,
$c = f_0 = EI = \omega = 1$,
$\bar{w} = 0.01$, and
$\intcc{t_0,t_f} = \intcc{0,20}$.
We aim at over-approximating the reachable tube
$\mathcal{R}(\intcc{t_0,t_{f}})$ of \ref{eq:LTVsystem} for several
instances of $N_d$, using the proposed method. The computed
over-approximations will be used to obtain bounds on the bridge
displacements and will be analyzed in terms of accuracy and
computational costs.

\subsection{Implementation}

To address the problem described above, we employ the zonotopic
variant of the proposed method as given in
\ref{ProposedMethod:Zonotopes}. Moreover, we take advantage of the
smoothness of the matrix-valued function $A(\cdot)$ in
\ref{eq:PDEmodelReducedMatrixAandInputU} and use a second order
approximation $\tilde{\phi}$ of the transition matrix as given in
Lemma \ref{lem:SecondOrderTaylorsMethod}. The zonotopic variant is
implemented in MATLAB (2019a), and MATLAB is run on an
AMD Ryzen 5 2500U/ 2GHz processor. Plots of zonotopes are produced
with the help of software CORA \cite{Althoff15}. 
\subsection{Results}

First, we demonstrate the over-approximations obtained by the proposed
method. \ref{fig:ExampleReachableTubeOverApprox} illustrates several
over-approximations of $\mathcal{R}(\intcc{t_0,t_f})$, with $N_{d}=4$.
As seen in the mentioned figure, the accuracy of the
over-approximations increases as the value of $N$ increases, which
matches with the findings of this work. Next, bounds on the
displacement of the bridge are obtained based on several
over-approximations of $\mathcal{R}(\intcc{t_0,t_f})$ for several
instances of $N_{d}$. \ref{fig:ExampleReachableTubeOverApprox} also
illustrates bounds on bridge displacement obtained for several
values of $N_d$ and $N$. The quality of the bounds improves as $N$
increases due to the increased accuracy of the computed
over-approximations.

Finally, the scalability of the proposed method is illustrated by
considering a fixed value of the time discretization parameter
$N$ and selected values of the spatial discretization parameter
$N_d$.
\ref{fig:ExamplePerformance} indicates a memory requirement of order
$O(n^2)$, as predicted by the discussion at the end of Section
\ref{s:MainResults}, and a run time of order $O(n^{2.4})$
approximately, which is less than the predicted $O(n^3)$.
The difference is due to the fact that MATLAB takes advantage of the
sparse structure of the matrix $\widetilde \Phi$ in Lemma
\ref{lem:SecondOrderTaylorsMethod}, inherited from the matrix $A$ in
\ref{eq:PDEmodelReducedMatrixAandInputU}.

\subsection{Comparison: Ellipsoidal Techniques}
\label{ss:Comparison}

\begin{figure}[t]
\begin{center}
\includegraphics[width=.49\linewidth,keepaspectratio=true]{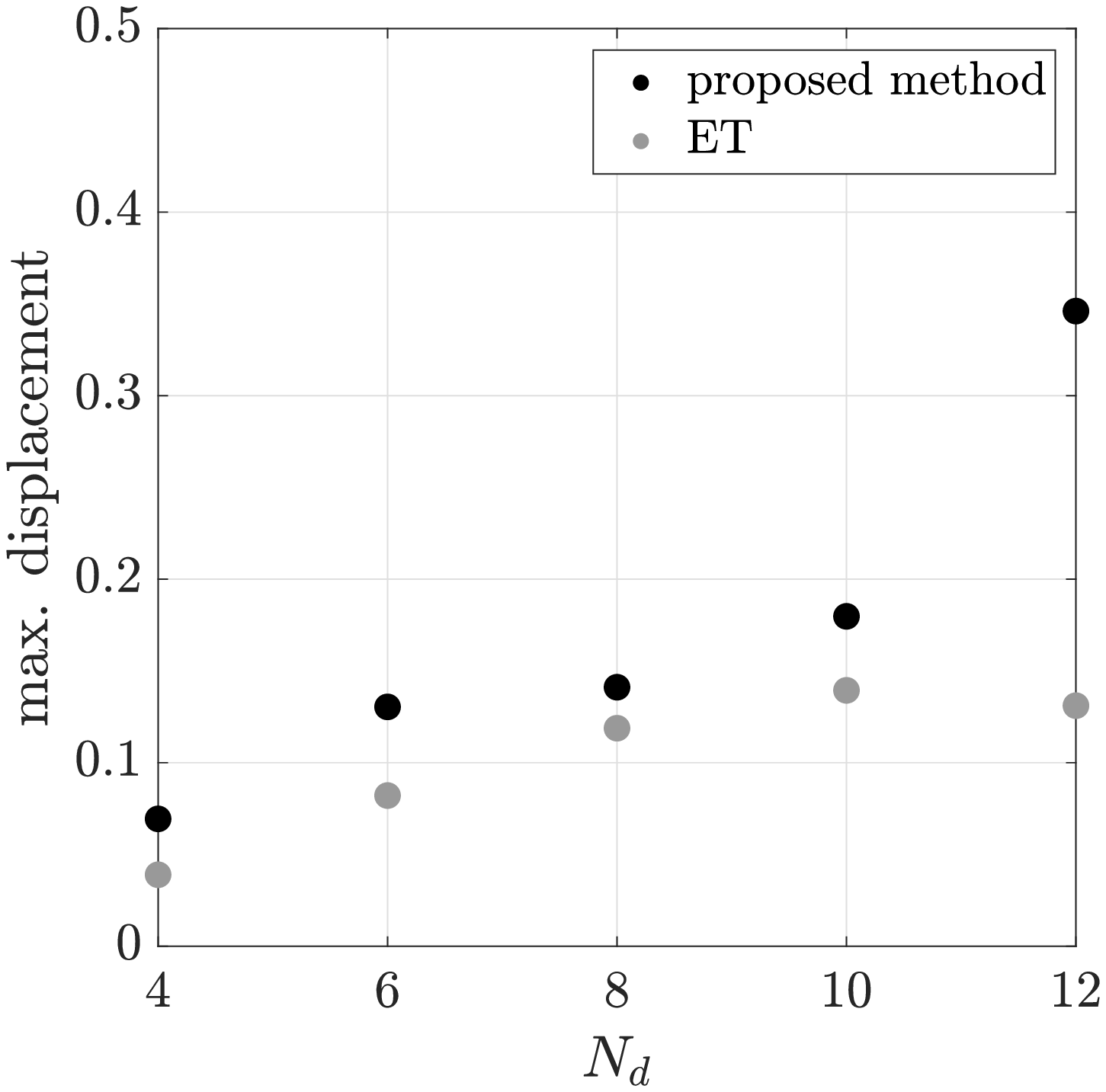}
\includegraphics[width=.49\linewidth,keepaspectratio=true]{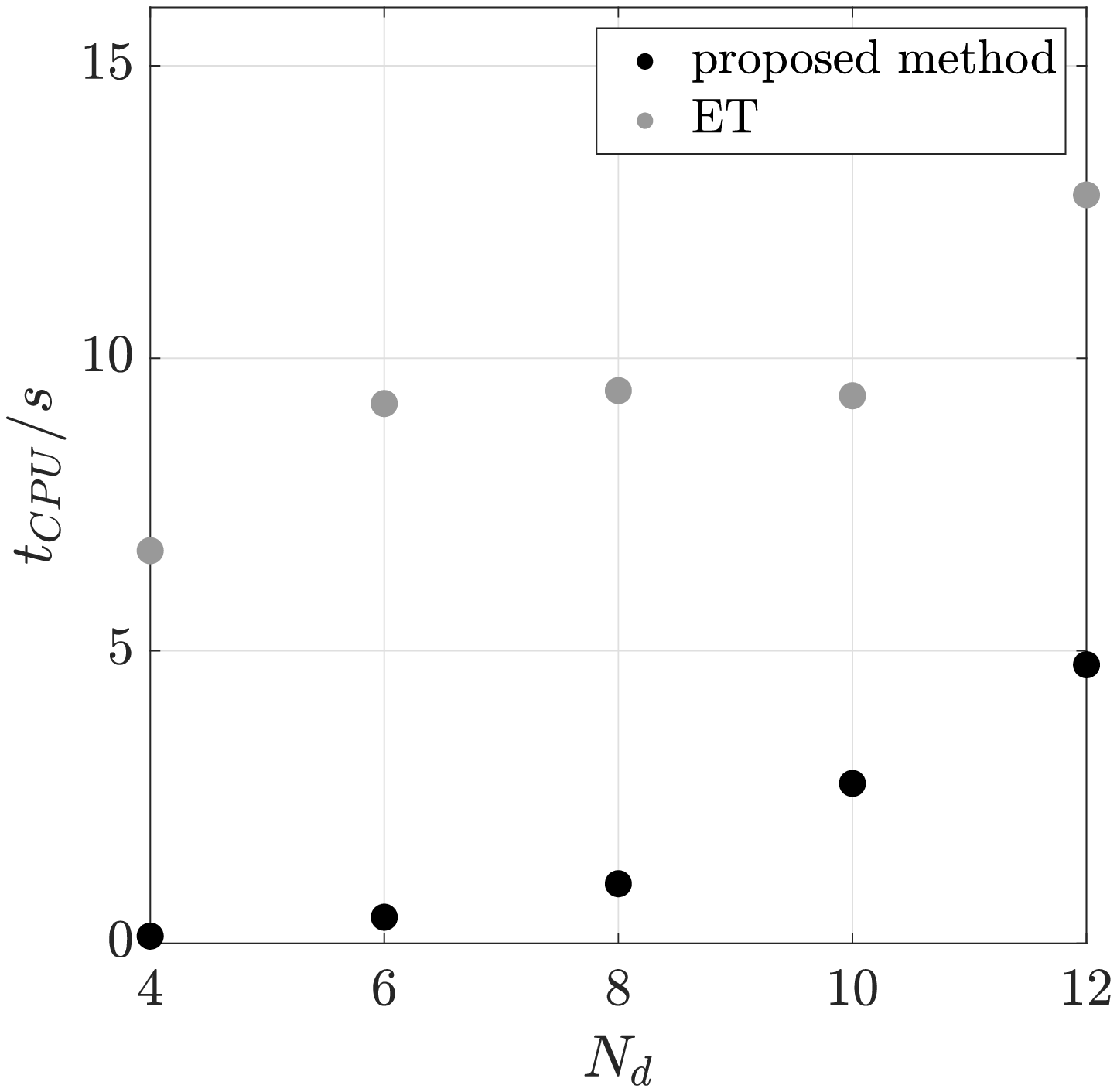}
\end{center}
\caption{\label{fig:Comparison}
Comparison between the zonotopic variant, with $N= 800$, and the ET both applied to the example in Section
\ref{s:examples} for different instances
of the discretization parameter $N_{d}$:
Estimation of maximum bridge displacement at time $t_{f}$, and run time (right).}
\end{figure}

In this subsection, we illustrate the performance of the zonotopic
variant of the proposed method in comparison with ellipsoidal
techniques \cite{KurzhanskiVaraiya14} implemented in the
ellipsoidal toolbox (ET) \cite{GagarinovKurzhanski14}.

As we have discussed in the Introduction, ellipsoidal techniques yield
reachable tubes only in implicit form,
and hence we restrict the scope of the comparison to reachable
sets only.
Here, we consider different instances of system
\ref{eq:PDEmodelReducedSystem} with $n\in \intcc{2;18}$, or
equivalently, $N_d \in \intcc{4;12}$. The ellipsoidal set
$\tilde U \defas ( \bar{w} / m ) \mathbb{B}_2$,
where $\mathbb{B}_2$ is the $2$-norm closed unit ball in
$\mathbb{R}^{N_d - 3}$, is considered as the input set when
applying the ET, as the ET is directly applicable to ellipsoidal
initial and input sets only.
Since $\tilde U \subseteq U$, the ET is given the advantage of using a
smaller input set.
As the sets $B \tilde U$ and $X_0$ are degenerate, the ET
requires defining a regularization parameter, which we have set to be
$10^{-3}$, which introduces full dimensional conservative
substitutes. Moreover, we arbitrarily use the direction vector
$e_1 = (1,0,0,\ldots)$ in our computations of ellipsoidal
approximations. The zonotopic variant is implemented with $N = 800$
and restricted to compute reachable sets only (computations in
equations \ref{ProposedMethod:Zonotopes:diJi} and
\ref{ProposedMethod:Zonotopes:Hi} are omitted). Both techniques
are set to obtain over-approximations of $\mathcal{R}(t_f)$ which
are subsequently used to estimate the maximum bridge displacement,
upon all nodal points, at time $t_f$.

\ref{fig:Comparison} (left) shows that for instances of system
\ref{eq:PDEmodelReducedSystem}, with $N_{d}\in \intcc{4;12}$, the
zonotopic variant performs very well in comparison with the ET in
terms of estimating maximum bridge displacement despite the inherent
disadvantage of using a larger input set and of accounting for
approximation errors which are not considered by the ET.
As seen from \ref{fig:Comparison} (left), the effect of these errors
is more pronounced for increasing state space dimension, which is due
to rapidly growing estimates of matrix norms of the system matrices
and their derivatives (growth is of order $O(n^{4})$ as a
result of the finite difference approximation of the $4$th order
derivative in equation \ref{eq:PDEmodel}).
\ref{fig:Comparison} (right) illustrates that the zonotopic variant
outperforms the ET in terms of computational time for the instances of
system \ref{eq:PDEmodelReducedSystem}, with $N_{d}\in \intcc{4;12}$.
We note, however, that the ET computes additionally
under-approximations of reachable sets, which might contribute to the
relatively higher computational time.

\section{Conclusion}
\label{s:Conclusion}

We have proposed a method to compute over-approximations of reachable
tubes for LTV systems that are additionally represented in a form
suitable for formal verification purposes. The method has been
inspired by existing techniques for LTI systems, and, when applied
to that special case, it is almost equivalent to
those in \cite{Girard05,LeGuernicGirard09}, except that it
additionally requires to repeatedly compute convex hulls.
We have also presented a zonotopic variant of the method and
demonstrated its performance on an example, which indicates that the
computational effort is comparable to that of existing methods
approximating reachable sets rather than tubes.
The accuracy of our method could be improved by
implementing component-wise estimates as in
e.g.~\cite{FrehseLeGuernicDonzeCottonRayLebeltelRipadoGirardDangMaler11},
in place of matrix and vector norms.

\appendix
\label{s:appendix}

\begin{lemma}[Convex Sets]
\label{lem:OperationsOnConvexSets}
Let $\Omega,\Gamma \subseteq \mathbb{R}^n$ be convex,
$\alpha, \beta \in \mathbb{R}$, and
$L \in \mathbb{R}^{m \times n}$. Then the following holds.
\begin{enumerate}
\item
\label{lem:OperationsOnConvexSets:ConvexCompact}
The sets $\alpha \Omega$, $\Omega + \Gamma$ and $L \Omega$ are convex.
They are additionally compact if $\Omega$ and $\Gamma$ are so.
\item
\label{lem:OperationsOnConvexSets:CHOfUnionOfConvexSets}
If $\Omega$ and $\Gamma$ are additionally non-empty, then
$
\conv( \Omega \cup \Gamma )
=
\Set{ \lambda x + ( 1 - \lambda ) y }{ x \in \Omega, y \in \Gamma, \lambda \in \intcc{0,1}}
$.
\qedhere
\end{enumerate}
\end{lemma}

\begin{proof}
For \ref{lem:OperationsOnConvexSets:ConvexCompact}, see
\cite[Sect.~3]{Rockafellar70}
and note that images of compact sets under
continuous maps are compact.
For \ref{lem:OperationsOnConvexSets:CHOfUnionOfConvexSets}, see
\cite[Th.~3.3]{Rockafellar70}.
\end{proof}
The \concept{Hausdorff distance} $d_H(\Omega,\Gamma)$ of two non-empty
bounded subsets $\Omega, \Gamma \subseteq \mathbb{R}^n$
w.r.t.~$\| \cdot \|$ is defined by
\[
d_H(\Omega,\Gamma)
=
\inf
\Set{ \varepsilon > 0}{%
\Omega \subseteq \Gamma + \varepsilon \mathbb{B},
\Gamma \subseteq \Omega + \varepsilon \mathbb{B}
},
\]
and is used to measure the extent by which the two sets $\Omega$ and
$\Gamma$ differ from each other. This distance satisfies the triangle
inequality, it is a metric when restricted to non-empty compact
subsets of $\mathbb{R}^n$ \cite{DeBlasi76}, and it additionally enjoys
the properties in the following lemma.

\begin{lemma}[Hausdorff Distance]
\label{lem:HausdorffDistance}
Let $\Omega, \Omega', \Gamma, \Gamma' \subseteq \mathbb{R}^n$ be
non-empty and bounded, and let $A,B \in \mathbb{R}^{m \times n}$ and
$\delta, \varepsilon \in \mathbb{R}_{+}$.
Then the following holds:

\begin{enumerate}
\item
\label{lem:HausdorffDistance:1}
$d_H( \Omega + \Gamma, \Omega' + \Gamma' )
\le
d_H( \Omega, \Omega' ) + d_H( \Gamma, \Gamma' )$.
\item
\label{lem:HausdorffDistance:2}
$
d_H( A \Omega, A \Gamma )
\le
\| A \| d_H( \Omega, \Gamma )
$.
\item
\label{lem:HausdorffDistance:3}
$
d_H( A \Omega, B \Omega )
\le
\| A-B \| \| \Omega \|$.
\item
\label{lem:HausdorffDistance:4}
$\| \Omega \| = d_H( \Omega, \{ 0 \} )$.
\item
\label{lem:HausdorffDistance:5}
If $\Omega$ and $\Gamma$ are additionally closed, then
$d_H( \Omega, \Gamma ) \le \varepsilon$
iff
$\Omega \subseteq \Gamma + \varepsilon \mathbb{B}$
and
$\Gamma \subseteq \Omega + \varepsilon \mathbb{B}$.
\item
\label{lem:HausdorffDistance:6}
If $d_H( \Omega, \Gamma ) \le \varepsilon$, then
$
d_H( \Omega, \Gamma + \delta \mathbb{B} )
\le
\varepsilon + \delta
$.
\item
\label{lem:HausdorffDistance:Union}
Let
$(\Omega_i)_{i\in I}$ and
$(\Gamma_i)_{i\in I}$ be families of non-empty subsets of
$\mathbb{R}^n$. Then
$
d_H
\left(
\cup_{i \in I} \Omega_i,
\cup_{i \in I} \Gamma_i
\right)
\le
\sup_{i \in I} d_H( \Omega_i, \Gamma_i )
$.
\qedhere
\end{enumerate}
\end{lemma}

\begin{proof}
For \ref{lem:HausdorffDistance:1}, \ref{lem:HausdorffDistance:3} and
\ref{lem:HausdorffDistance:5}, see \cite[Lemma 2.2]{DeBlasi76},
\cite[Lemma 0.1.2.7]{Baier95},
and the discussion in \cite[p.~48]{Schneider93}.
The definition of $d_H$ directly implies
\ref{lem:HausdorffDistance:2} and \ref{lem:HausdorffDistance:4}, and
\ref{lem:HausdorffDistance:1} and \ref{lem:HausdorffDistance:4} imply
\ref{lem:HausdorffDistance:6}.
To prove \ref{lem:HausdorffDistance:Union}, we may assume
that
$
\sup_{i \in I} d_H( \Omega_i, \Gamma_i )
<
\varepsilon
$
for some real $\varepsilon$.
Then $d_H(\Omega_i,\Gamma_i) < \varepsilon$ for every $i \in I$, and
in turn
$
\Omega_i
\subseteq
\varepsilon \mathbb{B} + \Gamma_i
$.
It follows that
$
\cup_{i \in I} \Omega_i
\subseteq
\varepsilon \mathbb{B} + \cup_{i \in I} \Gamma_i
$,
and similarly with the roles of $\Omega_i$ and $\Gamma_i$
interchanged. Hence,
$
d_H
\left(
\cup_{i \in I} \Omega_i,
\cup_{i \in I} \Gamma_i
\right)
\le
\varepsilon
$,
and since the bound $\varepsilon$ was
arbitrary, we are done.
\end{proof}

\noindent
In the following result, given a norm $\| \cdot \|$ on
$\mathbb{R}^n$, it is assumed that the norm of any matrix
$A \in \mathbb{R}^{n \times p}$ is w.r.t.~the maximum norm on
$\mathbb{R}^p$, i.e.,
$
\| A \|
=
\sup \Set{ \| A x \| }{ \| x \|_{\infty} \le 1 }
$.

\begin{lemma}[Zonotopes]
\label{lem:Zonotopes}
Let $b, c \in \mathbb{R}^n$,
$F \in \mathbb{R}^{n \times p}$,
$G \in \mathbb{R}^{n \times q}$,
and $L \in \mathbb{R}^{m \times n}$, and denote
$\Omega = \cZonotope(b,F)$,
$\Gamma = \cZonotope(c,G)$.
Then the following holds.
\begin{enumerate}
\item
\label{lem:Zonotopes:Addition}
$
\Omega + \Gamma
=
\cZonotope( b + c, (F,G) )
$,
where $(F,G) \in \mathbb{R}^{n \times (p+q)}$.
\item
\label{lem:Zonotopes:LinearTransformation}
$L \Gamma = \cZonotope(L c,L G)$.
\item
\label{lem:Zonotopes:Norm}
$
\| \Gamma \|
=
\| (c, G) \|$,
where $(c, G) \in \mathbb{R}^{n \times (q+1)}$.
In particular,
$
\| \Gamma \|_{\infty}
=
\max \Set{ |c_i| + \sum_{j=1}^q | G_{i,j} | }{ i \in \intcc{1;n} }
$.
\item
\label{lem:Zonotopes:Enclosure}
If $p = q$, then
$
\conv( \Omega \cup \Gamma )
\subseteq
\cZonotope
\left(
\Enclosure((b,F),(c,G))
\right)
$,
where the operator $\Enclosure$ is defined in
\ref{def:ZonotopeEnclosure:e}.
Moreover, the Hausdorff distance between the two sets does not exceed
$\| F - G \|$. In particular, for the Hausdorff distance w.r.t.~the
maximum norm, the bound equals
$
\max
\Menge{ \sum_{j=1}^q | F_{i,j} - G_{i,j} | }{ i \in \intcc{1;n} }
$.
\qedhere
\end{enumerate}
\end{lemma}

\begin{proof}
For \ref{lem:Zonotopes:Addition} and
\ref{lem:Zonotopes:LinearTransformation}, see
\cite[Prop.~1.4, 1.5]{LeStoicaAlamoCamachoDumur13}.
To prove \ref{lem:Zonotopes:Norm}, note that
$\| \cZonotope(0,(c,G)) \| = \| (c,G) \|$
by the very definition \ref{def:Zonotopes:e} of
$\cZonotope$, and that $\Gamma \subseteq \cZonotope(0,(c,G))$. It remains to show that
$\| \cZonotope(0,(c,G)) \| \le \| \Gamma \|$. To this end, let
$p = \alpha c + y$ for some $\alpha \in \intcc{-1,1}$ and
$y \in G \intcc{-1,1}^q$. Then
$p = \lambda (c+y) + ( 1 - \lambda ) (y-c)$ for
$\lambda = ( 1 + \alpha ) / 2 \in \intcc{0,1}$, and so
$\| p \| \le \max \{ \| c + y \|, \| c - y \| \}$ as $\| \cdot \|$ is
convex. It follows that $\| p \| \le \| \Gamma \|$ since
$c+y,c-y \in \Gamma$.
The set inclusion claim in \ref{lem:Zonotopes:Enclosure} is known
\cite{Girard05}; we sketch a proof: Denote
$E = \cZonotope ( \Enclosure( (b,F), (c,G) ) )$ and
let $x = b + F \lambda$ for some $\lambda \in \intcc{-1,1}^p$. Then
$x = (b+c)/2 + (b-c)/2 + (F+G) \lambda /2 + (F-G) \lambda /2 \in E$.
This shows that $\Omega \subseteq E$, and similarly we obtain
$\Gamma \subseteq E$. The claim follows as $E$ is convex.
To prove the estimate, which improves Girard's result \cite{Girard05},
let $x \in E$. Then
$
2 x
=
b + c
+
\alpha ( b - c )
+
( F + G ) \mu
+
( F - G ) \nu
$
for some $\alpha \in \intcc{-1,1}$ and
some $\mu, \nu \in \intcc{-1,1}^p$. Define
$
y
=
\lambda ( c + G \mu )
+ ( 1 - \lambda ) ( b + F \mu )
$
for $\lambda = ( 1 - \alpha ) / 2 \in \intcc{0,1}$.
Then $y \in \conv( \Omega \cup \Gamma )$ by Lemma
\ref{lem:OperationsOnConvexSets}\ref{lem:OperationsOnConvexSets:CHOfUnionOfConvexSets},
and
$
x - y
=
( F - G ) ( \nu - \alpha \mu ) / 2
$.
Since $\| \nu - \alpha \mu \|_{\infty} \le 2$, we arrive at
$\| x - y \| \le \| F - G \|$, which proves the bound.
\end{proof}

\begin{lemma}[Taylor's method of order $\mathbf 2$]
\label{lem:SecondOrderTaylorsMethod}
Suppose that hypotheses
\ref{i:ProblemStatement:Time}-\ref{i:ProblemStatement:Matrices} in
Section \ref{s:ProblemStatement} hold.
Additionally assume that $A$ is of class $C^2$ and that
$\| \ddot A (t) \| \le M_{\ddot A}$ holds for all
$t \in \intcc{t_0,t_{f}}$.
Define
$\widetilde{\phi} \colon D \to \mathbb{R}^{n \times n}$ by
$
\widetilde{\phi} (t,s)
=
\id + (t-s)A(s) + (t-s)^2 (\dot A(s) + A(s)^2)/2$.
Then condition \ref{i:ProblemStatement:ApproxPHI} in Section
\ref{s:ProblemStatement} holds with
$\theta$ given
by $\theta(h) = (1 + 3 M_{\dot A} / M_A^2 + M_{\ddot A} / M_A^3) ( \exp(h M_A) - h^2 M_A^2 / 2 - h M_A - 1 )$.
\end{lemma}

\begin{proof}
The map $\theta$ is clearly monotonically increasing, and
$\theta(h) = O(h^3)$ as $h \to 0$, which implies
\ref{i:ProblemStatement:ApproxPHI:LocalErrorOfOrder2}.
Given $s \in \intcc{t_0,t_f}$, Taylor's formula for
$\phi(\cdot,s)$ about the point $s$ reads
$
\phi(t,s)
=
\widetilde{\phi}(t,s)
+
\frac{1}{2}(t-s)^3 \int_0^1 (1-z)^2 D_1^3 \phi(s + z(t-s),s) dz
$.
Using the identity
$
D_1^3 \phi(t,s)
=
(\ddot A(t) + 2 A(t) \dot A(t) + \dot A(t) A(t) + A(t)^3) \phi(t,s)
$
as well as the estimate \ref{e:ExpEstimates} we obtain the estimate
\ref{i:ProblemStatement:ApproxPHI:LocalError}.
\end{proof}

\bibliographystyle{GReissigIEEEtran}
\bibliography{GReissigBSTCTL,GReissigPreambles,mrabbrev,GReissigStrings,GReissigFremde,GReissigFremdeTMP,GReissigCONF,GReissigJOURNALS,GReissigPATENT,GReissigREPORTS,GReissigTALKS,GReissigTHESES,referencesstudents}
\ifDraftVersion
\IfFileExists{./ReplyToReviewers.tex}{%
\ifCLASSOPTIONtwocolumn\noindent\onecolumn\large\fi
\clearpage
\include{./ReplyToReviewers}
}{}%
\fi
\end{document}